\documentclass[11pt,reqno,sumlimits]{amsart}

\usepackage[utf8]{inputenc}
\usepackage{amssymb, amscd, amsmath, epsfig, mathtools}
\usepackage{amsthm}
\usepackage{enumerate}
\usepackage{xcolor}
\usepackage{scalerel}
\usepackage{soul}
\usepackage{tikz-cd}

\usepackage[margin=1.0in]{geometry}

\newtheorem{theorem}{Theorem}[section]
\newtheorem*{theorem*}{Theorem}
\newtheorem{definition}{Definition}[section]
\newtheorem{corollary}{Corollary}[section]

\newtheorem{proposition}{Proposition}[section]
\theoremstyle{definition}
\newtheorem{remark}{Remark}[section]

\newcommand{\R}{\mathbb R}

\newcommand{\calC}{\mathcal C}
\newcommand{\calL}{\mathcal L}
\newcommand{\dvol}{ d\text{Vol}_{g}}

\begin{document}

\title[Prescribed Scalar Curvature Problem with Dirichlete Condition]{Prescribed Scalar Curvature Problem under Conformal Deformation of A Riemannian Metric with Dirichlet Boundary Condition}
\author[J. Xu]{Jie Xu}
\address{
Current Address: Institute for Theoretical Sciences, Westlake University, Hangzhou, Zhejiang Province, China}
\email{xujie67@westlake.edu.cn}
\address{
Department of Mathematics and Statistics, Boston University, Bosotn, MA, U.S.A.}
\email{xujie@bu.edu}

\date{}	

\maketitle

\begin{abstract} In this article, we first show that for all compact Riemannian manifolds with non-empty smooth boundary and dimension at least 3, there exists a metric, pointwise conformal to the original metric, with constant scalar curvature in the interior, and constant scalar curvature on the boundary by considering the boundary as a manifold of its own with dimension at least 2. We then show a series of prescribed scalar curvature results in the interior and on the boundary, with pointwise conformal deformation. These type of results is both an analogy and an extension of Kazdan and Warner's ``Trichotomy Theorem" on a different type of manifolds. The key step of these problems is to obtain a positive, smooth solution of a Yamabe equation with Dirichlet boundary conditions.
\end{abstract}

\section{Introduction}
In this article, we consider the conformal deformation of a Riemannian metric on a compact manifold $ (\bar{M}, g) $ with Dirichlet boundary conditions along non-empty smooth boundary $ \partial M $, $ n = \dim \bar{M} \geqslant 3 $. On one hand, we consider a conformal Riemannian metric with prescribed constant scalar curvature in the interior $ M $, this is the boundary Yamabe problem with Dirichlet conditions; on the other hand, given a smooth function $ S \in \calC^{\infty}(\bar{M}) $, we discuss the necessary and sufficient conditions of the existence of a conformal metric with prescribed scalar curvature $ S $. 

In Escobar problem \cite{ESC}, conformal deformation of a Riemannian metric with constant scalar curvature in the interior $ M $ and constant mean curvature on $ \partial M $ was introduced. This boundary Yamabe problem was fully solved very recently in \cite{XU4} for minimal boundary case, and in \cite{XU5} for general constant mean curvature case. For prescribed scalar and mean curvatures on general $ (\bar{M}, g) $, we refer to a very recent result \cite{XU6}. However, the mean curvature along $ \partial M $ is extrinsic by considering $ \partial M $ as part of the whole compact manifold with boundary $ (\bar{M}, g) $. Another point of view is to consider $ \partial M $ as a closed manifold of its own with the induced metric $ \imath^{*} g $ where $ \imath: \partial M \hookrightarrow \bar{M} $ is the inclusion map. Therefore we can consider the prescribed scalar curvature under the conformal deformation of a Riemannian metric $ g $ on $ \bar{M} $ associates with a given conformal deformation of the induced metric $ \imath^{*} g $ on $ \partial M $ for some prescribed scalar curvature on $ \partial M $.

Throughout the article, we denote $ n = \dim \bar{M} \geqslant 3 $ and $ p = \frac{2n}{n - 2} $. Let $ g_{0} $ be some metric on $ \partial M $, which is pointwise conformal to the metric $ \imath^{*} g $ determined by $ \phi \in \calC^{\infty}(\partial M) $. When $ n > 3 $, it can be expressed as $ g_{0} = \left( \phi^{\frac{n+1}{n-3}} \right)^{\frac{n - 2}{4}} \imath^{*} g $; when $ n = 3 $, there exists some smooth function $ f $ such that $ g_{0} = e^{2f} \imath^{*} g $ and $ \phi = \left( e^{2f} \right)^{\frac{n - 2}{4}} $. Mathematically, we look for a positive smooth function $ u \in \calC^{\infty}(\bar{M}) $ such that:

(A) the new metric $ \tilde{g} = u^{p-2} g $ admits a constant scalar curvature $ \lambda $ on interior $ M $, with $ u = \phi $ along $ \partial M $; 

(B) given a smooth function $ S \in \calC^{\infty}(\bar{M}) $, the new metric $ \tilde{g} = u^{p-2} g $ admits the scalar curvature $ S $ on $ M $. The problem (B) can be reduced to the existence of the solution of the PDE
\begin{equation}\label{intro:eqn1}
\begin{split}
-a\Delta_{g} u + S_{g} u & = S u^{p-1} \; {\rm in} \; M; \\
u & = \phi \; {\rm on} \; \partial M.
\end{split}
\end{equation}
Here $ a = \frac{4(n - 1)}{n - 2} $, $ S_{g} $ is the scalar curvature with respect to $ g $. The problem (A) is realized as the existence of the solution of the PDE (\ref{intro:eqn1}) with $ S = \lambda $ for some constant $ \lambda \in \R $. These problems are extensions of the results in \cite{XU2} in which we consider the Yamabe equation with constant Dirichlet boundary condition in a very small Riemannian domain $ (\Omega, g) $ of $ \R^{n}, n \geqslant 3 $.
\medskip

The first main result for prescribed constant scalar curvature is as follows:
\begin{theorem}\label{intro:thm1}
Let $ (\bar{M}, g) $ be a compact manifold with smooth boundary $ \partial M $, $ n = \dim \bar{M} $. Let
\begin{equation*}
\imath : \partial M \hookrightarrow M
\end{equation*}
be the inclusion map. Assume $ \partial M $ to be a $ (n - 1) $-dimensional closed manifold with induced metric. There exists a metric $ \tilde{g} $, conformal to $ g $, that admits a constant scalar curvature on the interior $ M $ of the $ n $-dimensional manifold $ \bar{M} $, and a constant scalar curvature on the $ (n - 1) $-dimensional closed manifold $ \partial M $ with respect to the induced metric $ \imath^{*} \tilde{g} $.
\end{theorem}
\medskip

To present the result below, we denote $ \eta_{1} $ to be the first eigenvalue of the conformal Laplacian on $ (\bar{M}, g) $ with Robin condition
\begin{equation*}
-\frac{4(n - 1)}{n - 2} \Delta_{g} u + S_{g} u = \eta_{1} u \; {\rm in} \; M, \frac{\partial u}{\partial \nu} + \frac{n-2}{2} h_{g} u = 0 \: {\rm on} \; \partial M.
\end{equation*}
$ S_{g} $ is the scalar curvature function and $ h_{g} $ is the mean curvature function. The second result is about prescribed scalar curvature:
\begin{theorem}\label{intro:thm2} Let $ (\bar{M}, g) $ be a compact manifold with smooth boundary $ \partial M $.

(i) If $ \eta_{1} < 0 $, any negative function $ S \in \calC^{\infty}(\bar{M}) $ is realized as a prescribed curvature function of some metric under pointwise conformal change of $ (\bar{M}, g) $; meanwhile the $ ( n - 1) $-dimensional closed manifold $ \partial M $ admits a constant scalar curvature with the same conformal change;

(ii) If $ \eta_{1} = 0 $, any function $ S \in \calC^{\infty}(\bar{M}) $ that is negative on a neighborhood of $ \partial M $ or $ S \equiv 0 $ is realized as a prescribed curvature function of some metric under conformal change of $ (\bar{M}, g) $; meanwhile the $ ( n - 1) $-dimensional closed manifold $ \partial M $ admits a constant scalar curvature with the same conformal change;

(iii) If $ \eta_{1} > 0 $, any function $ S \in \calC^{\infty}(\bar{M}) $ is realized as a prescribed curvature function of some metric under conformal change of $ (\bar{M}, g) $; meanwhile the $ ( n - 1) $-dimensional closed manifold $ \partial M $ admits a constant scalar curvature with the same conformal change;
\end{theorem}
We point out here that the result above is not only an analogy of Kazdan and Warner's ``Trichotomy Theorem'' \cite{KW3}, \cite{KW4}, \cite{KW} but also an extension of this classification in the interior of compact manifolds with boundary, since in ``Trichotomy Theorem", the existences of the metrics may not be obtained from pointwise conformal change. But in Theorem \ref{intro:thm2}, all conclusions are with the restriction of pointwise conformal change.
\medskip

Due to the Yamabe problem \cite{PL}, \cite{XU3} or the uniformization theorem on closed Riemann surfaces, $ \partial M $---as a closed manifold of $ \dim (\partial M ) \geqslant 2 $---admits a constant scalar curvature with respect to some conformal change of the induced metric $ \imath^{*} g $, i.e. the scalar curvature of $ g_{0} $ is constant, intrinsically. In this case, solving (\ref{intro:eqn1}) with $ S = \lambda $ implies the existence of a constant scalar curvature on $ M $ under conformal change associated with a constant scalar curvature on the $ (n - 1) $-dimensional closed manifold $ \partial M $.

We no longer consider the Yamabe invariant for this Dirichlet problem. Actually, (\ref{intro:eqn1}) is the Euler-Lagrange equation of the functional
\begin{equation}\label{intro:eqn2}
E(u) = \frac{\int_{M} a \lvert \nabla u \rvert^{2} \dvol + \int_{M} S_{g} u^{2} \dvol - \int_{\partial M} 2\frac{\partial u}{\partial \nu} \phi dS}{\left(\int_{M} u^{p} \dvol \right)^{\frac{2}{p}}}.
\end{equation}
Here $ dS $ is the volume form on $ \partial M $, $ \nu $ is the unit outward normal vector along $ \partial M $. Firstly, we cannot just minimize $ E(u) $ on $ u \in H^{1}(M, g) $ with $ u \geqslant 0 $ due to the last term on the numerator of (\ref{intro:eqn2}). This is quite different from the Yamabe problem or the Escobar problem. Furthermore, it is difficult to analyze the behavior of $ \frac{\partial u}{\partial \nu} $ along the boundary. In addition, since $ E(u) $ in (\ref{intro:eqn2}) is not a Yamabe invariant when $ n \geqslant 3 $, we may not be able to use many powerful tools established in solving the Yamabe problem and the boundary Yamabe problem in global calculus of variations. Therefore, it is natural to consider the solvability of (\ref{intro:eqn1}) by a direct method, here in particular, the monotone iteration scheme, as we used in \cite{XU6} in showing the prescribed scalar and mean curvatures on $ (\bar{M}, g) $.
\medskip

This article is organized as follows. In \S2, we introduce essential definitions, functional spaces, and results that will be used in later sections. In particular, results of the Escobar problems that were proved in \cite{XU4} were introduced, and will play a role in later context. In \S3, we prove two versions of monotone iteration schemes with respect to the Dirichlet problem (\ref{intro:eqn1}), one for globally constant function $ S $, and the other for general function $ S \in \calC^{\infty}(\bar{M}) $. In \S4, we consider the problem (A), i.e. the prescribed constant scalar curvature under conformal change with Dirichlet boundary condition. After several results for the relations between the signs of $ \eta_{1} $---the first eigenvalue of conformal Laplacian with Robin condition, and $ \tilde{\eta}_{1} $---the first eigenvalue of conformal Laplacian with Dirichlet condition from Proposition \ref{dirichlet:prop1} to Proposition \ref{dirichlet:prop3}, we solve the PDE (\ref{intro:eqn1}) with $ S = \lambda $ for some constant $ \lambda $ in Theorem \ref{dirichlet:thm1}, Theorem \ref{dirichlet:thm2}, Theorem \ref{dirichlet:thm3} classified by the sign of $ \eta_{1} $. In particular, Theorem \ref{dirichlet:thethm} states that under conformal change, the manifold $ \bar{M} $ admits a metric with constant scalar curvature in the interior $ M $, and a constant scalar curvature on $ \partial M $ as a manifold of its own. Some further results are given in the rest of \S4, classified by either the sign of $ \tilde{\eta}_{1} $ or the sign of $ S_{g} $. In \S5, we consider the PDE (\ref{intro:eqn1}) for general smooth function $ S $ and solve the problem (B). The sufficient conditions for $ S $ to be a prescribed scalar curvature in the interior $ M $ of the compact manifold $ \bar{M} $ are also classified by the sign of $ \eta_{1} $, as stated in Theorem \ref{scalar:thm1}, Theorem \ref{scalar:thm2} and Theorem \ref{scalar:thm3}. In Theorem \ref{scalar:thethm}, we conclude that adding some possible restriction on $ S $ if necessary, the function $ S $ is a prescribed scalar curvature of some metric conformal to $ g $; meanwhile, the boundary $ \partial M $ admits a constant scalar curvature of its own. Forget about the boundary behavior, we would like to point out that Theorem \ref{scalar:thm3}, Corollary \ref{scalar:cor2} and Corollary \ref{scalar:cor3} state that any function can be a prescribed scalar curvature in the interior of a compact manifolds $ (\bar{M}, g) $ with smooth boundary, provided that the first eigenvalue $ \eta_{1} $ is positive. This result is analogous to Kazdan and Warner's famous ``Trichotomy Theorem", which states that any function can be a prescribed scalar curvature of a closed manifold which admits a positive constant scalar curvature. Theorem \ref{scalar:thm2} and Corollary \ref{scalar:cor1} state the case when $ \eta_{1} = 0 $. Theorem \ref{scalar:thm1} states the case when $ \eta_{1} < 0 $. We point out that our conclusions are extensions of results in ``Trichotomy Theorem" by restricting our choices of metrics within the conformal class $ [g] $.
\medskip

\section{The Preliminaries} 
In this section, we list definitions and results required for later analysis, including Sobolev spaces, maximum principles, $ H^{s} $-type and $ W^{s, p} $-type elliptic regularities, results of the boundary Yamabe problem, Sobolev embedding, etc. Throughout this article, we set $ (\bar{M}, g) $ be a compact Riemannian manifold with non-empty smooth boundary $ \partial M $, $ n = \dim \bar{M} \geqslant 3 $. Here $ g $ is assumed to be extended smoothly to $ \partial M $. We also assume that $ \partial M $ is itself a closed $ (n - 1) $-dimensional manifold with outward normal unit vector $ \nu $. Throughout this article, we denote smooth function by $ \calC^{\infty} $, smooth function with compact support by $ \calC_{c}^{\infty} $ and continuous function by $ \calC^{0} $, etc.

\begin{definition}\label{pre:def1} Let $ (\bar{M}, g) $ be a compact Riemannian $n$-manifold with smooth boundary $ \partial M $ and volume density $\dvol$. Let $u$ be a real valued function. Let $ \langle v,w \rangle_g$ and $ |v|_g = \langle v,v \rangle_g^{1/2} $ denote the inner product and norm with respect to $g$. 

(i) 
For $1 \leqslant p < \infty $,
\begin{equation*}
\mathcal{L}^{p}(M, g) \; {\rm is\ the\ completion\ of}\ \left\{ u \in \calC^{\infty}(M) : \Vert u\Vert_{p,g}^p :=\int_{M} \left\lvert u \right\rvert^{p} d\text{Vol}_{g} < \infty \right\}.
\end{equation*}

(ii) For $\nabla u$  the Levi-Civita connection of $g$, 
and for $ u \in \calC^{\infty}(\Omega) $ or $ u \in \calC^{\infty}(M) $,
\begin{equation}\label{pre:eqn1}
\lvert \nabla^{k} u \rvert_g^{2} := (\nabla^{\alpha_{1}} \dotso \nabla^{\alpha_{k}}u)( \nabla_{\alpha_{1}} \dotso \nabla_{\alpha_{k}} u).
\end{equation}
\noindent In particular, $ \lvert \nabla^{0} u \rvert^{2}_g = \lvert u \rvert^{2} $ and $ \lvert \nabla^{1} u \rvert^{2}_g = \lvert \nabla u \rvert_{g}^{2}.$\\

(iii) For $ s \in \mathbb{N}, 1 \leqslant p < \infty $,
\begin{equation}\label{pre:eqn2}
W^{s, p}(M, g) = \left\{ u \in \mathcal{L}^{p}(M, g) : \lVert u \rVert_{W^{s, p}(M, g)}^{p} = \sum_{j=0}^{s} \int_{M} \left\lvert \nabla^{j} u \right\rvert^{p}_g \dvol < \infty \right\}.
\end{equation}
\noindent Here $ \lvert D^{j}u \rvert^{p} := \sum_{\lvert \alpha \rvert = j} \lvert \partial^{\alpha} u \rvert^{p} $ in the weak sense. Similarly, $ W_{0}^{s, p}(\bar{M}) $ is the completion of $ \calC_{c}^{\infty}(\bar{M}) $ with respect to the $ W^{s, p} $-norm. In particular, $ H^{s}(M, g) : = W^{s, 2}(M, g) $ are the usual Sobolev spaces. We similarly define $H_{0}^{s}(\bar{M}), H_{0}^{s}(\bar{M} ,g)$. 
\end{definition}

Dirichlet boundary value problem (\ref{intro:eqn1}) requires extending functions to $ \partial M $ in the trace sense, as shown in the following proposition.
\begin{proposition}\label{pre:prop1}\cite[Ch.4, Prop.~4.5]{T}
Let $ (\bar{M}, g) $ be a compact manifold with smooth boundary $ \partial M $. Let $ u \in H^{1}(M, g) $. Then there exists a bounded linear operator
\begin{equation*}
T : H^{s}(M, g) \rightarrow H^{s - 1}(\partial M, \imath^{*}g)
\end{equation*}
such that
\begin{equation}\label{boundary:eqn11}
\begin{split}
T u & = u \bigg|_{\partial M}, \; \text{if} \; u \in \calC^{\infty}(\bar{M}) \cap H^{s}(M, g); \\
\lVert T u \rVert_{H^{s-1}(\partial M, \imath^{*} g)} & \leqslant K'' \lVert u \rVert_{H^{s}(M, g)}.
\end{split}
\end{equation}
Here $ K'' $ only depends on $ (M, g), s $ and is independent of $ u $. Furthermore, the map $ T : H^{s}(M, g) \rightarrow H^{s - 1}(\partial M, \imath^{*} g) $ is surjective.
\end{proposition}
\medskip

Next we list the results of both $ H^{s} $-type and $ W^{2, p} $-type elliptic regularities on $ (\bar{M}, g) $ with respect to the second order differential operator $ -a\Delta_{g} + A $ with any positive constant $ A > 0 $. It is clear that $ -a\Delta_{g} + A $ is injective on $ \partial M $ with trivial Dirichlet boundary condition, due to Fredholm dichotomy. 
\begin{theorem}\label{pre:thm1}
Let $ (\bar{M}, g) $ be a compact manifold with smooth boundary $ \partial M $. Let $ Lu = -a\Delta_{g} u + Au $ be the second order elliptic operator with positive constant $ A > 0 $.

(i) \cite[Ch.5, Prop.~11.2]{T} Assume $ u \in \calL^{2}(M, g) $. If $ Lu \in \calL^{2}(M, g) $ and $ u \in H^{1}(\partial M, \imath^{*}g) $, then $ u \in H^{2}(M, g) $ with the estimates
\begin{equation}\label{pre:eqn3}
\lVert u \rVert_{H^{s + 2}(M, g)} \leqslant C_{1} \left( \lVert Lu \rVert_{H^{s}(M, g)} + \lVert u \rVert_{H^{s + 1}(\partial M, \imath^{*}g)} + \lVert u \rVert_{\calL^{2}(M, g)} \right).
\end{equation}
The constant $ C_{1} $ depends only on $ s, (\bar{M}, g), L $ and is independent of $ u $.

(ii) \cite[Thm.~2.1]{XU5} Let $ f \in \calL^{p}(M, g), \tilde{f} \in W^{2, p}(M, g) $. Assume $ u \in H^{1}(M, g) $ be a weak solution of the following boundary value problem
\begin{equation}\label{pre:eqn4}
L u = f \; {\rm in} \; M, B_{d} u: = u = \tilde{f} \; {\rm on} \; \partial M.
\end{equation}
If, in addition, $ u \in \calL^{p}(M, g) $, then $ u \in W^{2, p}(M, g) $ with the following estimates
\begin{equation}\label{pre:eqn5}
\lVert u \rVert_{W^{2, p}(M, g)} \leqslant C_{2} \left( \lVert Lu \rVert_{\calL^{p}(M, g)} + \lVert  B_{d}u \rVert_{W^{2, p}(M, g)} \right).
\end{equation}
Here $ C_{2} $ depends on $ L, p $ and the manifold $ (\bar{M}, g) $ and is independent of $ u $.
\end{theorem}
\begin{remark}\label{pre:re1}
(i) In \cite[Ch.5, Prop.~11.2]{T}, the last term in (\ref{pre:eqn3}) is $ \lVert u \rVert_{H^{1}(M, g)} $. This term can be replaced by $ \lVert u \rVert_{\calL^{2}(M, g)} $ due to the Peter-Paul inequality, see e.g. \cite[Prop.~2.4]{XU4}.

(ii) The PDE in Theorem 2.1 of \cite{XU5} is with the Robin condition and $ \tilde{f} \in W^{1, p}(M, g) $. Here the PDE in (\ref{pre:eqn4}) is with Dirichlet boundary condition. The proof here is essentially the same as in \cite{XU5}, both relying on the local results due to \cite[Thm.~15.1]{Niren4}.
\end{remark}
\medskip

A strong maximum principle on $ (\bar{M}, g) $ with respect to the operator $ Lu = -a\Delta_{g} u + Au $ with $ A > 0 $ is required, which will be applied to the monotone iteration scheme for Dirichlet boundary problem.
\begin{proposition}(Strong Maximum Principle)\label{pre:prop2}\cite[Ch.5, Prop.~2.6]{T}
Let $ (\bar{M}, g) $ be a compact manifold with smooth boundary $ \partial M $ and $ Lu = -a\Delta_{g} u + Au $ with constant $ A > 0 $. Assume $ u \in \calC^{2}(M) \cap \calC^{0}(\bar{M}) $. 

(i) If 
\begin{equation*}
Lu \geqslant 0 \; {\rm in} \;  M,   u \geqslant 0 \; {\rm on} \; \partial M,
\end{equation*}
then $ u \geqslant 0 $ on M;

(ii) If 
\begin{equation*}
Lu \leqslant 0 \; {\rm in} \;  M,   u \leqslant 0 \; {\rm on} \; \partial M,
\end{equation*}
then $ u \leqslant 0 $ on M.
\end{proposition}
\medskip

We set
\begin{equation}\label{pre:eqn6}
\Box_{g} u : = -a\Delta_{g} u + S_{g} u
\end{equation}
to be the conformal Laplacian. We discuss the eigenvalue problems of $ \Box_{g} $ with respect to both the Dirichlet condition and Robin condition.
\begin{proposition}\label{pre:prop3}\cite[Lemma~1.1]{ESC}
Let $ (\bar{M}, g) $ be a compact Riemannian manifold with smooth boundary $ \partial M $. Let $ S_{g}, h_{g} $ be the scalar curvature and mean curvature, respectively. The following eigenvalue problem 
\begin{equation}\label{pre:eqn7}
-a\Delta_{g} \varphi + S_{g} \varphi = \eta_{1} \varphi \; {\rm in} \; M, B_{g} \varphi : = \frac{\partial \varphi}{\partial \nu} + \frac{2}{p - 2} h_{g} \varphi = 0 \; {\rm on} \; \partial M.
\end{equation}
admits a real, smooth, positive solution $ \varphi \in \calC^{\infty}(\bar{M}) $. Here $ \eta_{1} $ is the first eigenvalue of conformal Laplacian $ \Box_{g} $ with Robin condition $ B_{g} u = 0 $.
\end{proposition}

\begin{proposition}\label{pre:prop4}
Let $ (\bar{M}, g) $ be a compact Riemannian manifold with smooth boundary $ \partial M $. Let $ S_{g} $ be the scalar curvature. The following eigenvalue problem 
\begin{equation}\label{pre:eqn8}
-a\Delta_{g} \tilde{\varphi} + S_{g} \tilde{\varphi} = \tilde{\eta}_{1} \tilde{\varphi} \; {\rm in} \; M, B_{d} \tilde{\varphi} : = \tilde{\varphi} = 0 \; {\rm on} \; \partial M.
\end{equation}
admits a real, smooth, positive solution $ \tilde{\varphi} \in \calC^{\infty}(\bar{M}) $. Here $ \tilde{\eta}_{1} $ is the first eigenvalue of conformal Laplacian $ \Box_{g} $ with Dirichlet condition $ B_{d} u = 0 $.
\end{proposition}
\begin{proof} It is standard that the first eigenvalue $ \tilde{\eta}_{1} $ is characterized by
\begin{equation*}
\tilde{\eta}_{1} = \inf_{u \geqslant 0, u \in H_{0}^{1}(M, g)} \frac{\int_{M} a \lvert \nabla_{g} u \rvert^{2} \dvol + \int_{M} S_{g} u^{2} \dvol}{\int_{M} u^{2} \dvol} : = \inf_{u \geqslant 0, u \in H_{0}^{1}(M, g)} E'(u).
\end{equation*}
Without loss of generality, we may characterize $ \tilde{\eta}_{1} $ among all $ u \in H_{0}^{1}(M, g)$ with $ \lVert u \rVert_{\calL^{2}(M, g)} = 1 $. By Poincar\'e inequality on compact manifolds with Dirichlet boundary condition \cite{LY}, we apply standard calculus of variation and it follows that the minimizing sequence $ \lbrace u_{k} \rbrace $ of $ E'(u) $ converges weakly in $ H^{1} $-norm, up to a subsequence, to some limit $ \tilde{\varphi} \geqslant 0 $ which solves (\ref{pre:eqn8}). Applying elliptic regularity and maximum principle, it follows that the smooth function $ \tilde{\varphi} > 0 $ on $ M $ solves (\ref{pre:eqn8}). The first eigenvalue is determined as $ \tilde{\eta}_{1} = E'(\tilde{\varphi}) $.
\end{proof}
\medskip

The boundary Yamabe equation with minimal boundary case
\begin{equation}\label{pre:eqn9}
\begin{split}
\Box_{g}u & : =  -a\Delta_{g} u + S_{g} u = \lambda u^{p-1} \; {\rm in} \; M; \\
B_{g} u & : = \frac{\partial u}{\partial \nu}  + \frac{2}{p-2} h_{g} u = 0 \; {\rm on} \; \partial M.
\end{split}
\end{equation}
plays an important role in solving the Dirichlet problem (\ref{intro:eqn1}). Here $ S_{g}, h_{g} $ are scalar and mean curvatures, respectively. As discussed in \cite{XU4}, the results are classified by the sign of first eigenvalue $ \eta_{1} $ in Proposition \ref{pre:prop3}.
\begin{theorem}\label{pre:thm2}\cite{XU4}
Let $ (\bar{M}, g) $ be a compact manifold with smooth boundary $ \partial M $, $ \dim \bar{M} \geqslant 3 $. Let $ \eta_{1} $ be the first eigenvalue of the boundary value problem $ \Box_{g} u = \eta_{1} u $ in $ M $, $ B_{g} u = 0 $ on $ \partial M $. Then
\begin{enumerate}[(i).]
\item If $ \eta_{1} = 0 $, then (\ref{pre:eqn9}) has a real, positive solution $ u \in \calC^{\infty}(\bar{M}) $ with $ \lambda = 0 $; 
\item If $ \eta_{1} < 0 $, then (\ref{pre:eqn9}) has a real, positive solution $ u \in \calC^{\infty}(\bar{M}) $ with $ \lambda < 0 $; 
\item If $ \eta_{1} > 0 $, then (\ref{pre:eqn9}) has a real, positive solution $ u \in \calC^{\infty}(\bar{M}) $ with $ \lambda > 0 $.
\end{enumerate}
\end{theorem}
\medskip

In addition, we also need results concerning the sign of mean curvature on $ \partial M $ under conformal change. 
\begin{theorem}\label{pre:thm3}
Let $ (\bar{M}, g) $ be a compact manifold with smooth boundary,

(i) there exists a conformal metric $ \tilde{g} \in [g] $ with mean curvature $ \tilde{h} > 0 $ everywhere on $ \partial M $;

(ii) there exists a conformal metric $ \tilde{g} \in [g] $ with mean curvature $ \tilde{h} < 0 $ everywhere on $ \partial M $.
\end{theorem}
\medskip

Sobolev embeddings are critical for the regularity of solution of (\ref{intro:eqn1}).
\begin{proposition}\label{pre:prop5}\cite[Ch.~2]{Aubin}  (Sobolev Embeddings) 
Let $ (\bar{M}, g) $ be a compact manifold with smooth boundary $ \partial M $.

(i) For $ s \in \mathbb{N} $ and $ 1 \leqslant p \leqslant p' < \infty $ such that
\begin{equation}\label{pre:eqn10}
   \frac{1}{p} - \frac{s}{n} \leqslant \frac{1}{p'},
\end{equation}
\noindent  $ W^{s, p}(M, g) $ continuously embeds into $ \mathcal{L}^{p'}(M, g) $ with the following estimates: 
\begin{equation}\label{pre:eqn11}
\lVert u \rVert_{\calL^{p'}(M, g)} \leqslant K \lVert u \rVert_{W^{s, p}(M, g)}.
\end{equation}

(ii) For $ s \in \mathbb{N} $, $ 1 \leqslant p < \infty $ and $ 0 < \alpha < 1 $ such that
\begin{equation}\label{pre:eqn12}
  \frac{1}{p} - \frac{s}{n} \leqslant -\frac{\alpha}{n},
\end{equation}
Then  $ W^{s, p}(M, g) $ continuously embeds in the H\"older space $ \calC^{0, \alpha}(\bar{M}) $ with the following estimates:
\begin{equation}\label{pre:eqn13}
\lVert u \rVert_{\calC^{0, \alpha}(\bar{M})} \leqslant K' \lVert u \rVert_{W^{s, p}(M, g)}.
\end{equation}

(iii) Both embeddings above are compact embeddings provided that the strict inequalities hold in (\ref{pre:eqn10}) and (\ref{pre:eqn12}), respectively.
\end{proposition} 
\medskip

\section{Monotone Iteration Scheme for Dirichlet Boundary Condition}
In this section, we show the existence of the solution of the Dirichlet problem (\ref{intro:eqn1}) by assuming the existence of sub-solution and super-solution of the boundary Yamabe equation with Dirichlet boundary condition, i.e. given a smooth, positive function $ \phi $ on $ \partial M $ and some constant $ \lambda \in \R $, we assume the existences of $ u_{-}, u_{+} \in \calC^{\infty}(M) $ such that
\begin{align*}
\Box_{g} u_{+} & \geqslant \lambda u_{+}^{p-1} \; {\rm on} \; M, B_{d} u_{+} \geqslant \phi \; {\rm on} \; \partial M; \\
\Box_{g} u_{-} & \leqslant \lambda u_{-}^{p-1} \; {\rm on} \; M, B_{d} u_{-} \leqslant \phi \; {\rm on} \; \partial M.
\end{align*}
In the recent results of closed manifolds and compact manifolds with boundary, sub- and super-solutions in the weak sense were enough. Here we need the sub- and super-solutions hold in the classical sense.
\begin{theorem}\label{iteration:thm1}
Let $ (\bar{M}, g) $ be a compact manifold with smooth boundary $ \partial M $. Let $ q > \dim \bar{M} $. Let $ \phi \in \calC^{\infty}(\partial M) $ be a positive function on $ \partial M $. Let $ \lambda \in \R $ be some constant. Suppose that there exist $ u_{-} \in \calC^{0}(\bar{M}) \cap \calC^{\infty}(M) $ and $ u_{+} \in \calC^{0}(\bar{M}) \cap \calC^{\infty}(M) $, $ 0 \leqslant u_{-} \leqslant u_{+} $, $ u_{-} \not\equiv 0 $ on $ \bar{M} $ such that
\begin{equation}\label{iteration:eqn1}
\begin{split}
-a\Delta_{g} u_{-} + S_{g} u_{-} - \lambda u_{-}^{p-1} & \leqslant 0 \; {\rm in} \; M,  u_{-} \leqslant \phi \; {\rm on} \; \partial M \\
-a\Delta_{g} u_{+} + S_{g} u_{+} - \lambda u_{+}^{p-1} & \geqslant 0 \; {\rm in} \; M, u_{+} \geqslant \phi \; {\rm on} \; \partial M
\end{split}
\end{equation}
holds in the classical sense. Then there exists a real, positive solution $ u \in \calC^{\infty}(M) \cap \calC^{1, \alpha}(\bar{M}) $ of
\begin{equation}\label{iteration:eqn2}
\Box_{g} u = -a\Delta_{g} u + S_{g} u = \lambda u^{p-1}  \; {\rm in} \; M, B_{d} u =  \phi \; {\rm on} \; \partial M.
\end{equation}
\end{theorem}
\begin{proof}
Fix some $ q > \dim \bar{M} $. Fix any $ \lambda > 0 $. Denote $ u_{0} = u_{+} $. Choose a constant $ A > 0 $ such that
\begin{equation}\label{iteration:eqn3}
-S_{g}(x) + \lambda (p - 1) u(x)^{p-2} + A > 0, \forall u(x) \in [\min_{\bar{M}} u_{-}(x), \max_{\bar{M}} u_{+}(x) ], \forall x \in \bar{M}
\end{equation}
pointwise. For the first step, consider the linear PDE
\begin{equation}\label{iteration:eqn4}
-a\Delta_{g} u_{1} + Au_{1} = Au_{0} - S_{g} u_{0} + \lambda u_{0}^{p-1}  \; {\rm in} \; M, u_{1} = \phi \; {\rm on} \; \partial M.
\end{equation}
Since $ u_{0} \in \calC^{\infty}(M) $, it follows from Lax-Milgram \cite[Ch.~6]{Lax} that (\ref{iteration:eqn4}) admits a weak solution $ u_{1} \in H^{1}(M, g) $. $ u_{0} \in \calC^{\infty}(M) $ implies that $ u_{0} \in H^{s}(M, g) $ for all $ s \in \mathbb{Z}_{+} $, it follows from $ H^{s}$-type regularity in Theorem \ref{pre:thm1} (i) that $ u_{1} \in H^{s} $ for all $ s \in \mathbb{Z}_{+} $ and thus $ u_{1} \in \calC^{\infty}(M) $ due to Sobolev embedding. $ u_{0} \in \calC^{\infty}(M) $ implies $ u_{0} \in \calL^{q}(M, g) $. By $ W^{2, p} $-type regularity in Theorem \ref{pre:thm1} (ii) that $ u_{1} \in W^{2, q}(M, g) $. Due to Sobolev embedding in Proposition \ref{pre:prop5} (ii), we conclude that $ u_{1} \in W^{2, q}(M, g) \cap \calC^{1, \alpha}(\bar{M}) $ for some $ \alpha \in (0, 1) $.

We show that $ 0 \leqslant u_{-} \leqslant u_{1} \leqslant u_{+} $ pointwise. By (\ref{iteration:eqn1}) we have
\begin{equation*}
-a\Delta_{g} u_{0} + Au_{0} \geqslant Au_{0} - S_{g} u_{0} + \lambda u_{0}^{p-1}  \; {\rm in} \; M, u_{0} \geqslant \phi \; {\rm on} \; \partial M.
\end{equation*}
Subtracting this by (\ref{iteration:eqn4}), it follows that
\begin{equation}\label{iteration:eqn5}
-a\Delta_{g} (u_{0} - u_{1}) + A(u_{0} - u_{1}) \geqslant 0 \; {\rm on} \; M, u_{0} - u_{1} \geqslant 0 \; {\rm on} \; \partial M.
\end{equation}
Since $ u_{0} - u_{1} \in \calC^{2}(M) \cap \calC^{0}(\bar{M}) $, it follows from the strong maximum principle in Proposition \ref{pre:prop2} that $ u_{0} - u_{1} \geqslant 0 $. By the same argument with sub-solution $ u_{-} $, we conclude that $ 0 \leqslant u_{-} \leqslant u_{1} \leqslant u_{+} = u_{0} $ on $ \bar{M} $. Inductively, we consider the PDE
\begin{equation}\label{iteration:eqn6}
-a\Delta_{g} u_{k} + Au_{k} = Au_{k- 1} - S_{g} u_{k - 1} + \lambda u_{k - 1}^{p-1}  \; {\rm in} \; M, u_{k - 1} = \phi \; {\rm on} \; \partial M.
\end{equation}
such that $ u_{k - 1} \in \calC^{\infty}(M) \cap W^{2, p}(M, g) \cap \calC^{1, \alpha}(\bar{M}) $ solves the previous step. In addition, we assume
\begin{equation}\label{iteration:eqn7}
0 \leqslant u_{-} \leqslant u_{k - 1} \leqslant \dotso \leqslant u_{1} \leqslant u_{0}.
\end{equation}
By Lax-Milgram, Sobolev embedding, $ H^{s} $-type and $ W^{s, p} $-type regularities, we conclude that $ u_{k} \in \calC^{\infty}(M) \cap W^{2, q}(M, g) \cap \calC^{1, \alpha}(\bar{M}) $ solves (\ref{iteration:eqn6}). Taking subtraction between two adjacent iteration steps, we have
\begin{equation}\label{iteration:eqn8}
\begin{split}
-a\Delta_{g} (u_{k - 1} - u_{k}) + A(u_{k - 1} - u_{k}) & = A(u_{k - 2} - u_{k - 1}) - S_{g} (u_{k- 2} - u_{k - 1}) + \lambda \left(u_{k - 2}^{p-1} - u_{k - 1}^{p-1} \right) \; {\rm on} \; M; \\
u_{k - 1} - u_{k} & = 0 \; {\rm on} \; \partial M.
\end{split}
\end{equation}
Due to the choice of $ A $ in (\ref{iteration:eqn3}), we conclude that the right side of (\ref{iteration:eqn8}) satisfies
\begin{equation*}
A(u_{k - 2} - u_{k - 1}) - S_{g} (u_{k- 2} - u_{k - 1}) + \lambda \left(u_{k - 2}^{p-1} - u_{k - 1}^{p-1} \right)  \geqslant 0
\end{equation*}
due to $ u_{k - 2} \geqslant u_{k - 1} $ and mean value theorem. By maximum principle, we conclude that $ u_{k - 1} - u_{k} \geqslant 0 $ on $ M $. Since $ u_{k - 1} \geqslant u_{0} $, it follows that $ u_{k} \geqslant u_{-} $ due to the same argument above. Hence
\begin{equation*}
0 \leqslant u_{-} \leqslant \dotso \leqslant u_{k} \leqslant u_{k - 1} \leqslant \dotso \leqslant u_{1} \leqslant u_{0}, \forall k \in \mathbb{Z}_{+}.
\end{equation*}
We now show the convergence of the sequence $ \lbrace u_{k} \rbrace $. Since $ \phi \in \calC^{\infty}(\partial M) $, there exists a smooth function $ \Phi \in \calC^{\infty}(\bar{M}) $ which admits $ \phi $ on $ \partial M $. By $ W^{s, p} $-type regularity, we have
\begin{equation}\label{iteration:eqn8a}
\lVert u_{k} \rVert_{W^{2, q}(M, g)} \leqslant C \left( \left\lVert Au_{k - 1} - S_{g} u_{k - 1} + \lambda u_{k - 1}^{p-1} \right\rVert_{\calL^{p}(M, g)} + \lVert \Phi \rVert_{W^{2, p}(M, g)} \right), \forall k \in \mathbb{Z}_{+}.
\end{equation}
According to (\ref{iteration:eqn7}), we have
\begin{align*}
& \left\lVert Au_{k - 1} - S_{g} u_{k - 1} + \lambda u_{k - 1}^{p-1} \right\rVert_{\calL^{q}(M, g)} \\
& \qquad \leqslant A \sup_{\bar{M}} u_{+} \text{Vol}_{g}(M)^{\frac{1}{p}} + \sup_{\bar{M}} \lvert S_{g} \rvert \sup_{\bar{M}} u_{+} \text{Vol}_{g}(M)^{\frac{1}{p}} + \lvert \lambda \rvert \sup_{\bar{M}} u_{+}^{p-1} \text{Vol}_{g}(M)^{\frac{1}{p}}.
\end{align*}
It follows from (\ref{iteration:eqn8a}) that
\begin{equation*}
\lVert u_{k} \rVert_{W^{2, q}(M, g)} \leqslant C_{3}, \forall k \in \mathbb{Z}_{+}.
\end{equation*}
Due to Sobolev embedding, we have
\begin{equation}\label{iteration:eqn9}
\lVert u_{k} \rVert_{\calC^{1, \alpha}(\bar{M})} \leqslant C_{3}, \forall k \in \mathbb{Z}_{+}
\end{equation}
for some uniform upper bound $ C_{3} $. Due to Arzela-Ascoli, it follows that there exists a subsequence of $ \lbrace u_{k} \rbrace $ which converges uniformly to some limit $ u $. By monotonicity of $ \lbrace u_{k} \rbrace_{k \in \mathbb{Z}_{\geqslant 0}} $, the whole sequence $ \lbrace u_{k} \rbrace $ converges uniformly to $ u $. It follows that $ \lim_{k \rightarrow \infty} u_{k} = u $ in $ W^{2, q} $-sense. By Sobolev embedding, $ u \in \calC^{1, \alpha}(\bar{M}) $ and hence local Schauder estimates implies that $ u \in \calC^{2, \alpha}(M) $. Taking limit on both sides of (\ref{iteration:eqn6}), we conclude that
\begin{equation*}
-a\Delta_{g} u + S_{g} u = \lambda u^{p-1}  \; {\rm in} \; M, u = \phi \; {\rm on} \; \partial M.
\end{equation*}
Finally the standard bootstrapping method implies that $u \in \calC^{\infty}(M) \cap \calC^{1, \alpha}(\bar{M}) $. In addition, we have
\begin{equation*}
0 \leqslant u_{-} \leqslant u \leqslant u_{+} \; {\rm on} \; M.
\end{equation*}
\end{proof}
\medskip

The conclusion above still holds when we replace the constant $ \lambda $ by some smooth function $ S \in \calC^{\infty}(\bar{M}) $, as the following result states.
\begin{corollary}\label{iteration:cor1}
Let $ (\bar{M}, g) $ be a compact manifold with smooth boundary $ \partial M $. Let $ q > \dim \bar{M} $. Let $ \phi \in \calC^{\infty}(\partial M) $ be a positive function on $ \partial M $. Let $ S \in \calC^{\infty}(\bar{M}) $ be a smooth function on $ \bar{M} $. Suppose that there exist $ u_{-} \in \calC^{0}(\bar{M}) \cap \calC^{\infty}(M) $ and $ u_{+} \in \calC^{0}(\bar{M}) \cap \calC^{\infty}(M) $, $ 0 \leqslant u_{-} \leqslant u_{+} $, $ u_{-} \not\equiv 0 $ on $ \bar{M} $ such that
\begin{equation}\label{iteration:eqn10}
\begin{split}
-a\Delta_{g} u_{-} + S_{g} u_{-} -S u_{-}^{p-1} & \leqslant 0 \; {\rm in} \; M,  u_{-} \leqslant \phi \; {\rm on} \; \partial M \\
-a\Delta_{g} u_{+} + S_{g} u_{+} - S u_{+}^{p-1} & \geqslant 0 \; {\rm in} \; M, u_{+} \geqslant \phi \; {\rm on} \; \partial M
\end{split}
\end{equation}
holds in the classical sense. Then there exists a real, positive solution $ u \in \calC^{\infty}(M) \cap \calC^{1, \alpha}(\bar{M}) $ of
\begin{equation}\label{iteration:eqn11}
\Box_{g} u = -a\Delta_{g} u + S_{g} u = S u^{p-1}  \; {\rm in} \; M, B_{d} u =  \phi \; {\rm on} \; \partial M.
\end{equation}
\end{corollary}
\begin{proof} Replacing $ \lambda $ by $ S $ in the argument of Theorem \ref{iteration:thm1}, and choose the constant $ A > 0 $ such that
\begin{equation}\label{iteration:eqn12}
-S_{g}(x) + S(x) (p - 1) u(x)^{p-2} + A > 0, \forall u(x) \in [\min_{\bar{M}} u_{-}(x), \max_{\bar{M}} u_{+}(x) ], \forall x \in \bar{M}
\end{equation}
pointwise. The rest are exactly the same as above.
\end{proof}
\medskip 

\section{The Boundary Yamabe Problem with Dirichlet Boundary Condition}
In this section, we show the existence of a real, positive, smooth solution of the following boundary Yamabe equation 
\begin{equation}\label{dirichlet:eqn1}
\begin{split}
-a\Delta_{g} u + S_{g} u & = \lambda u^{p-1} \; {\rm on} \; M; \\
u & = \phi \; {\rm on} \; \partial M
\end{split}
\end{equation}
with Dirichlet boundary condition on compact manifold $ (\bar{M}, g) $ with smooth boundary $ \partial M $, on which a positive boundary value $ \phi \in \calC^{\infty}(\partial M) $ is given. Here $ \lambda \in \R $, which will be determined later, is the constant scalar curvature with respect to the metric $ \tilde{g} = u^{p-2} g $. Analogous to the results in \cite{XU4, XU5, XU3}, we analyze the solvability of (\ref{dirichlet:eqn1}) in several cases, classified by sign of $ \tilde{\eta}_{1} $ or $ \eta_{1} $, the first eigenvalue of conformal Laplacian with respect to Dirichlet or Robin conditions, respectively.

When $ \tilde{\eta}_{1} < 0 $, the original scalar curvature $ S_{g} $ should be negative somewhere with respect to $ g $. To see this, assume that $ S_{g} \geqslant 0 $ everywhere on $ \bar{M} $, the eigenvalue problem with Dirichlet condition says
\begin{equation*}
-\Delta_{g} \tilde{\varphi}  + S_{g} \tilde{\varphi}  = \tilde{\eta}_{1} \tilde{\varphi}  \Rightarrow -\Delta_{g} \tilde{\varphi}  = \left(-S_{g} + \tilde{\eta}_{1} \right) \tilde{\varphi}  < 0
\end{equation*}
in $ M $. Here $ \tilde{\varphi} > 0 $ in $ M $ is the corresponding eigenfunction. By maximum principle, it follows that $ \tilde{\varphi} $ achieves its maximum on $ \partial M $, on which $ \tilde{\varphi} \equiv 0 $. But $ \tilde{\varphi} > 0 $ in $ M $, contradiction. Unlike the Yamabe problem on closed manifolds and the boundary Yamabe problem on compact manifolds with boundary, there is almost no restriction for the scalar curvature $ S_{g} $ when $ \tilde{\eta}_{1} > 0 $.

Due to next result from Escobar \cite{ESC}, the sign of $ \eta_{1} $ is a Yamabe invariant. Hence $ S_{g} $ must be negative somewhere when $ \eta_{1} < 0 $, and $ S_{g} $ must be positive somewhere when $ \eta_{1} > 0 $.
\begin{proposition}\label{dirichlet:prop1}\cite[Prop.~1.3]{ESC}
Let $ (\bar{M}, g) $ be a compact manifold with smooth boundary. Assume $ g' = u^{p -2} g $ is a conformal metric to $ g $. If $ \eta_{1} $ and $ \eta_{1}' $ are the first eigenvalues of conformal Laplacian $ \Box_{g} $ and $ \Box_{g'} $ with respect to Robin condition in  (\ref{pre:eqn7}), respectively, then $ \text{sgn}(\eta_{1}) = \text{sgn}(\eta_{1}') $, or $ \eta_{1} = \eta_{1}' = 0 $.
\end{proposition}
By similar argument, we can show that the sign of $ \tilde{\eta}_{1} $ is also invariant under conformal change.
\begin{proposition}\label{dirichlet:prop2}\cite[Prop.~1.3]{ESC}
Let $ (\bar{M}, g) $ be a compact manifold with smooth boundary. Assume $ g' = u^{p -2} g $ is a conformal metric to $ g $. If $ \tilde{\eta}_{1} $ and $ \tilde{\eta}_{1}' $ are the first eigenvalues of conformal Laplacian $ \Box_{g} $ and $ \Box_{g'} $ with respect to Dirichlet condition in  (\ref{pre:eqn8}), respectively, then $ \text{sgn}(\tilde{\eta}_{1}) = \text{sgn}(\tilde{\eta}_{1}') $, or $ \tilde{\eta}_{1} = \tilde{\eta}_{1}' = 0 $.
\end{proposition}
\begin{proof} Let $ \tilde{\varphi} $ and $ \tilde{\varphi}' $ be the first eigenfunctions with respect to $ \tilde{\eta}_{1} $ and $ \tilde{\eta}_{1}' $, respectively. Due to the same argument in Proposition 1.3 of \cite{ESC}, we have
\begin{equation*}
\int_{M} a\lvert \nabla_{g} f \rvert^{2} \dvol + \int_{M} S_{g} f^{2} \dvol = \int_{M} a\lvert \nabla_{g} \left(u^{-1} f \right) \rvert^{2} d\text{vol}_{g'} + \int_{M} S_{g'} \left(u^{-1} f \right)^{2} d\text{vol}_{g'}
\end{equation*}
under conformal change $ g' = u^{p-2} g $ with any function $ f \in \calC_{c}^{\infty}(\bar{M}) $. Note that this holds since $ f = 0 $ on $ \partial M $ so there is no boundary term under integration by parts. By characterization of first eigenvalue of eigenvalue problem (\ref{pre:eqn8}) in Proposition \ref{pre:prop4}, we have
\begin{align*}
\tilde{\eta}_{1} \int_{M} \tilde{\varphi}^{2} \dvol & = \int_{M} a\lvert \nabla_{g} \tilde{\varphi} \rvert^{2} \dvol + \int_{M} S_{g} \tilde{\varphi}^{2} \dvol \\
& = \int_{M} a\lvert \nabla_{g} \left(u^{-1}  \tilde{\varphi} \right) \rvert^{2} d\text{vol}_{g'} + \int_{M} S_{g'} \left(u^{-1}  \tilde{\varphi} \right)^{2} d\text{vol}_{g'} \\
& \geqslant \tilde{\eta}_{1}' \int_{M} \left( u^{-1} \tilde{\varphi} \right)^{2} d\text{vol}_{g'}.
\end{align*}
By the same argument, we have
\begin{equation*}
\tilde{\eta}_{1}' \int_{M} \left( \tilde{\varphi}' \right)^{2} d\text{vol}_{g'} \geqslant \tilde{\eta}_{1} \int_{M} \left( u^{-1} \tilde{\varphi}' \right)^{2} \dvol.
\end{equation*}
The conclusion thus follows immediately.
\end{proof}
There exists some connection between signs of $ \eta_{1} $ and $ \tilde{\eta}_{1} $, as next proposition shows.
\begin{proposition}\label{dirichlet:prop3}
Let $ (\bar{M}, g) $ be a compact Riemannian manifold with smooth boundary $ \partial M $. Let $ \eta_{1}, \tilde{\eta}_{1} $ be the first eigenvalues of conformal Laplacian $ \Box_{g} $ with Robin and Dirichlet conditions, respectively. Then the following inequalities hold:
\begin{equation}\label{dirichlet:eqn2}
\tilde{\eta}_{1} \leqslant 0 \Rightarrow \eta_{1} \leqslant \tilde{\eta}_{1} \leqslant 0; \eta_{1} \geqslant 0 \Rightarrow \tilde{\eta}_{1} \geqslant \eta_{1} \geqslant 0.
\end{equation}
Furthermore, $ \eta_{1} $ and $ \tilde{\eta}_{1} $ cannot be both zero simultaneously, i.e.
\begin{equation}\label{dirichlet:eqn3}
\eta_{1} = 0 \Rightarrow \tilde{\eta}_{1} > 0; \tilde{\eta}_{1} = 0 \Rightarrow \eta_{1} < 0.
\end{equation}
\end{proposition}
\begin{proof}
Recall the characterizations of $ \eta_{1} $ and $ \tilde{\eta}_{1} $
\begin{align*}
\eta_{1} & = \inf_{u \in H^{1}(M, g) } \frac{\int_{M} a \lvert \nabla_{g} u \rvert^{2}\dvol + \int_{M} S_{g} u^{2} \dvol + \int_{\partial M} \frac{2}{p-2} h_{g} u^{2} dS }{\int_{M} u^{2} \dvol}; \\
\tilde{\eta}_{1} & = \inf_{u \in H_{0}^{1}(M, g) } \frac{\int_{M} a \lvert \nabla_{g} u \rvert^{2}\dvol + \int_{M} S_{g} u^{2} \dvol }{\int_{M} u^{2} \dvol}.
\end{align*}
If $ \tilde{\eta}_{1} < 0 $ with eigenfunction $ \tilde{\varphi} $, note that $ \tilde{\varphi} = 0 $ on $ \partial M $, it follows that
\begin{equation}\label{dirichlet:eqn4}
\begin{split}
0 \geqslant \tilde{\eta}_{1} & = \frac{\int_{M} a \lvert \nabla_{g} \tilde{\varphi} \rvert^{2}\dvol + \int_{M} S_{g} \tilde{\varphi}^{2} \dvol }{\int_{M} \tilde{\varphi}^{2} \dvol} \\
& = \frac{\int_{M} a \lvert \nabla_{g} \tilde{\varphi} \rvert^{2}\dvol + \int_{M} S_{g} \tilde{\varphi}^{2} \dvol + \int_{\partial M} \frac{2}{p-2} h_{g} \tilde{\varphi}^{2} dS }{\int_{M} \tilde{\varphi}^{2} \dvol} \geqslant \eta_{1}.
\end{split}
\end{equation}
Hence $ \eta_{1} \leqslant \tilde{\eta}_{1} \leqslant 0 $. When $ \eta_{1} > 0 $, we have
\begin{align*}
\tilde{\eta}_{1} & = \frac{\int_{M} a \lvert \nabla_{g} \tilde{\varphi} \rvert^{2}\dvol + \int_{M} S_{g} \tilde{\varphi}^{2} \dvol }{\int_{M} \tilde{\varphi}^{2} \dvol} \\
& = \frac{\int_{M} a \lvert \nabla_{g} \tilde{\varphi} \rvert^{2}\dvol + \int_{M} S_{g} \tilde{\varphi}^{2} \dvol + \int_{\partial M} \frac{2}{p-2} h_{g} \tilde{\varphi}^{2} dS }{\int_{M} \tilde{\varphi}^{2} \dvol} \geqslant \eta_{1} \geqslant 0.
\end{align*}
Therefore $ \tilde{\eta}_{1} \geqslant \eta_{1} \geqslant 0 $. A further observation of (\ref{dirichlet:eqn4}) implies that
\begin{equation}\label{dirichlet:eqn5}
\tilde{\eta}_{1} = \frac{\int_{M} a \lvert \nabla_{g} \tilde{\varphi} \rvert^{2}\dvol + \int_{M} S_{g} \tilde{\varphi}^{2} \dvol + \int_{\partial M} \frac{2}{p-2} h_{g} \tilde{\varphi}^{2} dS }{\int_{M} \tilde{\varphi}^{2} \dvol} > \eta_{1}
\end{equation}
for all cases since it is well-known that the eigenspace with respect to the first eigenvalue is one-dimensional. Thus (\ref{dirichlet:eqn3}) follows from (\ref{dirichlet:eqn5}).
\end{proof}
\medskip

The conformal invariances of the signs of $ \eta_{1} $ and $ \tilde{\eta}_{1} $ indicate that the existence of solution of the boundary Yamabe problem (\ref{dirichlet:eqn1}) with Dirichlet boundary condition could be classified by the sign of either $ \eta_{1} $ or $ \tilde{\eta}_{1} $. As in \cite{XU4, XU5}, existences of solutions of (\ref{dirichlet:eqn1}) with given $ \phi > 0 $ on $ \partial M $ and appropriate choices of $ \lambda $ is determined by the sign of $ \eta_{1} $. We start with the case $ \eta_{1} = 0 $.
\begin{theorem}\label{dirichlet:thm1}
Let $ (\bar{M}, g) $ be a compact Riemannian manifold with smooth boundary $ \partial M $. Let $ \phi > 0 $ be a smooth function on $ \partial M $. If $ \eta_{1} = 0 $, then (\ref{dirichlet:eqn1}) has a real solution $ u \in \calC^{\infty}(M) $, $ u > 0 $ on $ \bar{M} $ with $ \lambda = 0 $.
\end{theorem}
\begin{proof} By assumption of $ (\bar{M}, g) $, the scalar curvature is bounded on $ \bar{M} $, i.e.
\begin{equation*}
K_{1} \leqslant S_{g} \leqslant K_{2} \; {\rm on} \; \bar{M}
\end{equation*}
for some constants $ K_{1}, K_{2} $. Choose a positive constant $ C > 0 $ such that
\begin{equation*}
C \geqslant \max \lbrace \lvert K_{1} \rvert, \lvert K_{2} \rvert \rbrace.
\end{equation*}
Consider the PDE
\begin{equation}\label{dirichlet:eqn6}
-a\Delta_{g} u + Cu = 0 \; {\rm in} \; M, u = \phi \; {\rm on} \; \partial M.
\end{equation}
By standard elliptic theory, it follows that (\ref{dirichlet:eqn6}) admits a unique solution $ u_{1} \in H^{1}(M, g) $. By standard elliptic regularity, it follows that $ u_{1} \in \calC^{\infty}(M) \cap \calC^{0}(\bar{M}) $. Since $ C > 0 $, maximum principle says that $ \min_{\bar{M}} u_{1} = \min_{\partial M} \phi > 0 $. Lastly we see that
\begin{equation*}
-a\Delta_{g} u_{1} + S_{g} u_{1} \leqslant  -a\Delta_{g} u_{1} + \lvert S_{g} \rvert u_{1} \leqslant -a\Delta_{g} u_{1} + \max \lbrace \lvert K_{1} \rvert, \lvert K_{2} \rvert \rbrace u_{1} \leqslant -a\Delta_{g} u_{1} +C u_{1} = 0.
\end{equation*}
It follows that $ u_{1} > 0 $ is a sub-solution of (\ref{dirichlet:eqn1}) with $ \lambda = 0 $. Since $ \eta_{1} = 0 $, the corresponding eigenfunction $ \varphi > 0 $ on $ \bar{M} $ satisfies
\begin{equation*}
-a\Delta_{g} \varphi + S_{g}\varphi = 0 \; {\rm in} \; M, \frac{\partial \varphi} {\partial \nu} + \frac{2}{p-2} h_{g} \varphi = 0 \; {\rm on} \; M.
\end{equation*}
Since any scale of $ \varphi $ also solves the eigenvalue problem, we set $ u_{1}' : = \delta \varphi $, $ \delta > 0 $ large enough so that
\begin{equation*}
\min_{\bar{M}} u_{1}' \geqslant \max_{\bar{M}} u_{1}, \min_{\partial M} u_{1}' \geqslant \max_{\partial M} \phi.
\end{equation*}
Fix this $ \delta $. It follows that
\begin{equation*}
-a\Delta_{g} u_{1}' + S_{g} u_{1}' = 0 \; {\rm in} \; M, u_{1}' \geqslant \phi \; {\rm on} \; \partial M.
\end{equation*}
It follows from the choice of $ \delta $ that $ u_{1}' \geqslant u_{1} > 0 $ is a super-solution of (\ref{dirichlet:eqn1}) with $ \lambda = 0 $. Since both $ u_{1}, u_{1}' \in \calC^{0}(\bar{M}) \cap \calC_{c}^{\infty}(M) $, Theorem \ref{iteration:thm1} implies the existence of $ u \in \calC^{\infty}(M) $ with $ 0 < u_{1} \leqslant u \leqslant u_{1}' $ which solves (\ref{dirichlet:eqn1}).
\end{proof}
\begin{remark}\label{dirichlet:re1} $ \eta_{1} = 0 $ implies $ \tilde{\eta}_{1} > 0 $ due to Proposition \ref{dirichlet:prop3}, thus it is immediate by Fredholm dichotomy that
\begin{equation*}
-a\Delta_{g} u + S_{g} u = 0 \; {\rm in} \; M, u = \phi \; {\rm on} \; \partial M
\end{equation*}
admits a solution, which is exactly the unique solution of (\ref{dirichlet:eqn1}) when $ \lambda = 0 $. Unfortunately, we cannot examine the sign of this solution due to the lack of information with respect to $ S_{g} $. Thus we apply monotone iteration scheme to show the positivity of $ u $ in Theorem \ref{dirichlet:thm1}.
\end{remark}
\medskip

Next we consider the case when $ \eta_{1} < 0 $. Historically this is an easier case for the Yamabe problem and the boundary Yamabe problem. We apply monotone iteration scheme in Theorem \ref{iteration:thm1} again.
\begin{theorem}\label{dirichlet:thm2}
Let $ (\bar{M}, g) $ be a compact Riemannian manifold with smooth boundary $ \partial M $. Let $ \phi > 0 $ be a smooth function on $ \partial M $. If $ \eta_{1} < 0 $, then (\ref{dirichlet:eqn1}) has a real solution $ u \in \calC^{\infty}(M) $, $ u > 0 $ on $ \bar{M} $ with some $ \lambda < 0 $.
\end{theorem}
\begin{proof} Let $ \varphi $ be the eigenfunction with respect to $ \eta_{1} $, it follows that the smooth function $ \varphi > 0 $ on $ \bar{M} $ solves the PDE
\begin{equation*}
-a\Delta_{g} \varphi + S_{g} \varphi = \eta_{1} \varphi \; {\rm in} \; M, \frac{\partial \varphi}{\partial \nu} + \frac{2}{p - 2} h_{g} \varphi = 0 \; {\rm on} \; \partial M.
\end{equation*}
Since any scale of $ \varphi $ is still an eigenfunction, we set $ u_{2} = \delta \varphi $ with $ \delta $ small enough such that
\begin{equation*}
u_{2} \leqslant \min_{\partial M} \phi.
\end{equation*}
Thus $ u_{2} > 0 $ satisfies
\begin{equation}\label{dirichlet:eqn7}
-a\Delta_{g} u_{2} + S_{g} u_{2}= \eta_{1}u_{2} \; {\rm in} \; M, u_{2} \leqslant \phi \; {\rm on} \; \partial M.
\end{equation}
We choose $ \lambda < 0 $ such that
\begin{equation*}
0 > \lambda \geqslant \eta_{1} \frac{\min_{\bar{M}} u_{2}}{\max_{\bar{M}} u_{2}^{p-1}} \Leftrightarrow \eta_{1} \min_{\bar{M}} u_{2} \leqslant \lambda \max_{\bar{M}} u_{2}^{p-1}.
\end{equation*}
Fix this $ \lambda $. Since both $ \eta_{1} $ and $ \lambda $ are negative, it follows that
\begin{equation*}
\eta_{1} u_{2} \leqslant \eta_{1} \min_{\bar{M}} u_{2} \leqslant \lambda \max_{\bar{M}} u_{2}^{p-1} \leqslant \lambda u_{2}^{p-1}.
\end{equation*}
Therefore
\begin{equation*}
-a\Delta_{g} u_{2} + S_{g} u_{2} \leqslant \lambda u_{2}^{p-1} \; {\rm in} \; M, u_{2} \leqslant \phi \; {\rm on} \; \partial M.
\end{equation*}
Hence $ u_{2} $ is a sub-solution of (\ref{dirichlet:eqn1}) with the fixed $ \lambda < 0 $. For super-solution, we choose a constant $ C > 0 $ large enough such that
\begin{equation*}
C \geqslant \max_{\bar{M}} u_{2}, \min_{\bar{M}} S_{g} \geqslant \lambda C^{p-2}, C \geqslant \max_{\partial M} \phi.
\end{equation*}
This can be done since $ \lambda < 0 $ and $ S_{g} $ is smooth up to $ \partial M $ and hence is bounded below. Set $ u_{2}' = C $, it follows that
\begin{equation*}
-a\Delta_{g} u_{2}' + S_{g} u_{2}' \geqslant \lambda u_{2}'^{p-1} \; {\rm in} \; M, u_{2}' \geqslant \phi \; {\rm on} \; \partial M.
\end{equation*}
Hence $ u_{2}' $ is a super-solution of (\ref{dirichlet:eqn1}) with the fixed $ \lambda < 0 $ satisfies $ 0 < u_{2} \leqslant u_{2}' $ on $ \partial M $. By Theorem \ref{iteration:thm1}, there exist a real function $ u \in \calC^{\infty}(M) $ solves (\ref{dirichlet:eqn1}) with the fixed $ \lambda < 0 $ , which satisfies
\begin{equation*}
0 < u_{2} \leqslant u \leqslant u_{2}'.
\end{equation*}
\end{proof}
\medskip

Lastly we consider the case $ \eta_{1} > 0 $.
\begin{theorem}\label{dirichlet:thm3}
Let $ (\bar{M}, g) $ be a compact Riemannian manifold with smooth boundary $ \partial M $. Let $ \phi > 0 $ be a smooth function on $ \partial M $. If $ \eta_{1} > 0 $, then (\ref{dirichlet:eqn1}) has a real solution $ u \in \calC^{\infty}(M) $, $ u > 0 $ on $ \bar{M} $ with some $ \lambda > 0 $.
\end{theorem}
\begin{proof} By the same reason, we have
\begin{equation*}
K_{1} \leqslant S_{g} \leqslant K_{2} \; {\rm on} \; \bar{M}
\end{equation*}
for some constants $ K_{1}, K_{2} $. Choose $ C > 0 $ large enough so that
\begin{equation*}
C \geqslant \max \lbrace \lvert K_{1} \rvert, \lvert K_{2} \rvert \rbrace.
\end{equation*}
Consider the PDE
\begin{equation}\label{dirichlet:eqn8}
-a\Delta_{g} u + Cu = 0 \; {\rm in} \; M, u = \phi \; {\rm on} \; \partial M.
\end{equation}
By the same reason as in Theorem \ref{dirichlet:thm1}, (\ref{dirichlet:eqn8}) admits a real, positive solution $ u_{3} \in \calC^{\infty}(M) \cap \calC^{0}(\bar{M}) $.

Now we determine the choice of $ \lambda $. Let $ \varphi $ be the eigenfunction with respect to $ \eta_{1} $ satisfying
\begin{equation}\label{dirichlet:eqn9}
-a\Delta_{g} \varphi + S_{g} \varphi = \eta_{1} \varphi \; {\rm in} \; M, \frac{\partial \varphi}{\partial \nu} + \frac{2}{p - 2} h_{g} \varphi = 0 \; {\rm on} \; \partial M.
\end{equation}
Denote $ u_{3}' = \delta \varphi $ with positive scalar $ \delta > 0 $ large enough so that
\begin{equation*}
\min_{\bar{M}} u_{3}' \geqslant \max_{\partial M} u_{3}, \min_{\partial M} u_{3}' \geqslant \max_{\partial M} \phi.
\end{equation*}
Fix this $ \delta $. It is straightforward that $ u_{3}' $ also solves (\ref{dirichlet:eqn9}). Choose $ \lambda > 0 $ small enough so that
\begin{equation}\label{dirichlet:eqn10}
0 < \lambda \leqslant \eta_{1} \frac{\min_{\bar{M}} u_{3}'}{\max_{\bar{M}} \left( u_{3}' \right)^{p-1}} \Leftrightarrow \eta_{1} \min_{\bar{M}} u_{3}' \geqslant \lambda \max_{\bar{M}} \left( u_{3}' \right)^{p-1}.
\end{equation}
Fix this $ \lambda $. Since both $ \eta_{1} $ and $ \lambda $ are positive, we have
\begin{align*}
-a\Delta_{g} u_{3}' + S_{g} u_{3}' & = \eta_{1} u_{3}' \geqslant \eta_{1} \min_{\bar{M}} u_{3}' \geqslant \lambda \max_{\bar{M}} \left( u_{3}' \right)^{p-1} \geqslant \lambda \left( u_{3}' \right)^{p-1} \; {\rm in} \; M; \\
u_{3}' & \geqslant \phi \; {\rm on} \; \partial M.
\end{align*}
It follows that $ u_{3}' $ is a super-solution of (\ref{dirichlet:eqn1}) with this fixed $ \lambda > 0 $. By (\ref{dirichlet:eqn8}), positivity of $ u_{3} $ and the choice of $ C $, we have
\begin{align*}
-a\Delta_{g} u_{3} + S_{g} u_{3} & \leqslant -a\Delta_{g} u_{3} +  \max \lbrace \lvert K_{1} \rvert, \lvert K_{2} \rvert \rbrace u_{3} \leqslant -a\Delta_{g} u_{3} + C u_{3} = 0 \leqslant \lambda u_{3}^{p-1} \; {\rm in} \; M; \\
u_{3} & = \phi \; {\rm on} \; \partial M.
\end{align*}
In conclusion, $ u_{3} $ is a sub-solution of (\ref{dirichlet:eqn1}) with the same $ \lambda > 0 $. Furthermore, $ 0 < u_{3} \leqslant u_{3}' $ due to the choice of $ u_{3}' $. Clearly $ u_{3}, u_{3}' \in \calC^{\infty}(M) \cap \calC^{0}(\bar{M}) $, thus by Theorem \ref{iteration:thm1}) the boundary Yamabe equation (\ref{dirichlet:eqn1}) admits a positive solution $ u \in \calC^{\infty}(M) $, $ 0 < u_{3} \leqslant u \leqslant u_{3}' $ with the fixed $ \lambda > 0 $.
\end{proof}
\begin{remark} In (\ref{dirichlet:eqn9}), we use the PDE of the eigenvalue problem to obtain a super-solution. Alternatively, we can apply the result of Escobar problem obtained in \cite{XU4}, which is Theorem \ref{pre:thm2} here, to get a super-solution. To see this, we know that
\begin{equation*}
-a\Delta_{g} u + S_{g} u = D u^{p-1} \; {\rm in} \; M, \frac{\partial u}{\partial \nu} + \frac{2}{p-2} h_{g} u = 0 \; {\rm on} \; \partial M
\end{equation*}
with some positive constant $ D $ is solved by some solution $ u > 0 $ on $ \bar{M} $. We take $ u_{3}' = \delta u $ for $ \delta > 0 $ large enough so that
\begin{equation*}
\min_{\partial M} u_{3}' \geqslant \max_{\partial M} \phi.
\end{equation*}
After the scaling, we see that $ u_{3}' $ satisfies
\begin{equation*}
-a\Delta_{g} u_{3}' + S_{g} u_{3}' = \delta^{2 - p} D (u_{3}')^{p-1} \; {\rm in} \; M, \frac{\partial u_{3}'}{\partial \nu} + \frac{2}{p-2} h_{g} u_{3}' = 0 \; {\rm on} \; \partial M.
\end{equation*}
Therefore we can choose $ \lambda \in(0, \delta^{2-p} D] $, and the rest of the argument in Theorem \ref{dirichlet:thm3} follows. We can apply the result of Theorem \ref{pre:thm2} to other cases by using corresponding versions.
\end{remark}
\medskip

Combining Theorem \ref{dirichlet:thm1}, \ref{dirichlet:thm2} and \ref{dirichlet:thm3}, we reach our first capstone of this article.
\begin{theorem}\label{dirichlet:thethm}
Let $ (\bar{M}, g) $ be a compact manifold with smooth boundary $ \partial M $, $ n = \dim \bar{M} $. Let
\begin{equation*}
\imath : \partial M \hookrightarrow M
\end{equation*}
be the inclusion map. Assume $ \partial M $ to be a $ (n - 1) $-dimensional closed manifold with induced metric. There exists a metric $ \tilde{g} $, conformal to $ g $, that admits a constant scalar curvature on the interior $ M $ of the $ n $-dimensional manifold $ \bar{M} $, and a constant scalar curvature on the $ (n - 1) $-dimensional closed manifold $ \partial M $ with respect to the induced metric $ \imath^{*} \tilde{g} $.
\end{theorem}
\begin{proof} When $ n \geqslant 4 $, there exists a smooth function $ \tilde{\phi} > 0 $ on the $ (n - 1) $-dimensional manifold $ \partial M $ such that
\begin{equation*}
g_{0} = \tilde{\phi}^{\frac{n - 1 + 2}{n - 1 - 2}} \left(\imath^{*} g \right) = \tilde{\phi}^{\frac{n + 1}{n - 3}} \left(\imath^{*} g \right) : = \tilde{\phi}^{p'} \left(\imath^{*} g \right), n \geqslant 4
\end{equation*}
admits a constant scalar curvature on $ \partial M $, due to the result of the Yamabe problem. When $ n = 3 $, there exists a smooth function $ f $ on $ \partial M $ such that
\begin{equation*}
g_{0}' = e^{2f} \left(\imath^{*} g \right), n = 3
\end{equation*}
admits a constant Gaussian curvature, due to the uniformization theorem. Denote
\begin{equation}\label{dirichlet:eqn11}
\phi = \begin{cases} \tilde{\phi}^{\frac{p'}{p - 2}}, & n \geqslant 4 \\ \left(e^{2f} \right)^{\frac{1}{p - 2}}, & n = 3 \end{cases}.
\end{equation}
Fix this $ \phi \in \calC^{\infty}(\partial M ) $ when $ n \geqslant 3 $. By (\ref{dirichlet:eqn11}) we have $ \phi > 0 $ everywhere on $ \partial M $. According to either Theorem \ref{dirichlet:thm1}, \ref{dirichlet:thm2} or \ref{dirichlet:thm3}, depending on the sign of the first eigenvalue $ \eta_{1} $, there exists a smooth function $ u > 0 $ on $ \bar{M} $, $ u = \phi $ on $ \partial M $ such that $ \tilde{g} = u^{p-2} g $ admits a constant scalar curvature on $ M $. The induced metric of $ \tilde{g} $ is of the form
\begin{equation*}
\imath^{*} \tilde{g} = \imath^{*} \left (u^{p-2} g \right) = \phi^{p-2} \imath^{*} g = \begin{cases} \tilde{\phi}^{p'} \imath^{*} g = g_{0}, & n \geqslant 4 \\ e^{2f} \imath^{*} g = g_{0}', & n = 3 \end{cases}.
\end{equation*}
Hence the induced metric $ \imath^{*} \tilde{g} $ admits a constant scalar curvature on $ (n - 1) $-dimensional closed manifold $ \partial M $ simultaneously.
\end{proof}
\medskip

We can also consider $ \partial M $ as part of $ \bar{M} $, and as the next result follows.
\begin{corollary}\label{dirichlet:thecor}
Let $ (\bar{M}, g) $ be a compact manifold with smooth boundary $ \partial M $, $ n = \dim \bar{M} $. Let
\begin{equation*}
\imath : \partial M \hookrightarrow M
\end{equation*}
be the inclusion map. Assume $ \partial M $ to be a $ (n - 1) $-dimensional closed manifold with induced metric. There exists a metric $ \tilde{g} $, conformal to $ g $, that admits a constant scalar curvature on the interior $ M $ of the $ n $-dimensional manifold $ \bar{M} $; meanwhile the metric is unchanged along $ \partial M $.
\end{corollary}
\begin{proof} We choose $ \phi = 1 $ on $ \partial M $. The rest are the same as above.
\end{proof}
\medskip

Theorem \ref{dirichlet:thm1} through \ref{dirichlet:thm3} are classified by the sign of $ \eta_{1} $. Due to the relation between $ \eta_{1} $ and $ \tilde{\eta}_{1} $ in Proposition \ref{dirichlet:prop3}, we can check the solvability of (\ref{dirichlet:eqn1}) in terms of the sign of $ \tilde{\eta}_{1} $ with appropriate choices of $ \lambda $. The first result is given below when $ \tilde{\eta}_{1} < 0 $.
\begin{corollary}\label{dirichlet:cor1}
Let $ (\bar{M}, g) $ be a compact Riemannian manifold with smooth boundary $ \partial M $. Let $ \phi > 0 $ be a smooth function on $ \partial M $. If $ \tilde{\eta}_{1} < 0 $, then (\ref{dirichlet:eqn1}) has a real solution $ u \in \calC^{\infty}(M) $, $ u > 0 $ on $ \bar{M} $ with some $ \lambda < 0 $.
\end{corollary}
\begin{proof} By Proposition \ref{dirichlet:prop1}, $ \tilde{\eta}_{1} < 0 $ implies $ \eta_{1} < 0 $, the rest follows by Theorem \ref{dirichlet:thm2}.
\end{proof}
\medskip

Unlike $ \tilde{\eta}_{1} < 0 $ situation, we can not say much about the case $ \tilde{\eta}_{1} > 0 $. A partial result is given below.
\begin{theorem}\label{dirichlet:thm4}
Let $ (\bar{M}, g) $ be a compact Riemannian manifold with smooth boundary $ \partial M $. Let $ \phi > 0 $ be a smooth function on $ \partial M $. If $ S_{g} > 0 $ everywhere on $ \bar{M} $, then (\ref{dirichlet:eqn1}) has a real solution $ u \in \calC^{\infty}(M) $, $ u > 0 $ on $ \bar{M} $ with some $ \lambda > 0 $.
\end{theorem}
\begin{proof} Due to the characterization of $ \tilde{\eta}_{1} $,
\begin{equation*}
\tilde{\eta}_{1}  = \inf_{u \in H_{0}^{1}(M, g) } \frac{\int_{M} a \lvert \nabla_{g} u \rvert^{2}\dvol + \int_{M} S_{g} u^{2} \dvol }{\int_{M} u^{2} \dvol}.
\end{equation*}
By hypothesis $ S_{g} > 0 $, it follows that $ \tilde{\eta}_{1} > 0 $. Choose a constant $ C > 0 $ such that
\begin{equation*}
C \geqslant \max_{\bar{M}} S_{g}.
\end{equation*}
Consider the PDE
\begin{equation}\label{dirichlet:eqn12}
-a\Delta_{g} u + Cu = 0 \; {\rm in} \; M, u = \phi \; {\rm in} \; \partial M.
\end{equation}
Due to the same argument, it follows that (\ref{dirichlet:eqn12}) admits a positive, smooth solution $ u_{4} \in \calC^{\infty}(M) \cap \calC^{0}(\partial M) $. By eigenvalue problem, the eigenfunction $ \tilde{\varphi} $ with $ \tilde{\eta}_{1} $ satisfies
\begin{equation}\label{dirichlet:eqn13}
-a\Delta_{g} \tilde{\varphi} + S_{g} \tilde{\varphi} = \tilde{\eta}_{1} \tilde{\varphi} \; {\rm in} \; M, \tilde{\varphi} = 0 \; {\rm on} \; \partial M.
\end{equation}
Choose a constant $ C' > 0 $ such that
\begin{equation}\label{dirichlet:eqn14}
C' \geqslant \max_{\partial M} \phi, C' \geqslant \max_{\bar{M}} u_{4}.
\end{equation}
Define $ u_{4}' : = \tilde{\varphi} + C' $. Choose $ \lambda > 0 $ small enough such that
\begin{equation*}
C' \cdot \min_{\bar{M}} S_{g} \geqslant \lambda \max_{\bar{M}} \left( u_{4}' \right)^{p-1}.
\end{equation*}
Fix this $ \lambda > 0 $. This can be done since $ S_{g} > 0 $ everywhere by assumption and $ u_{4}' $ is bounded above. By choices of $ C' $, $ \lambda $ and positivity of $ S_{g} $, we have
\begin{align*}
-a\Delta_{g} u_{4}' + S_{g} u_{4}' & = -a\Delta_{g} (\tilde{\varphi} + C') + S_{g} (\tilde{\varphi} + C') = -a\Delta_{g} \tilde{\varphi} + S_{g} \tilde{\varphi} + S_{g} C' \\
& = \tilde{\eta}_{1} \tilde{\varphi} + S_{g} C' \geqslant  \lambda \max_{\bar{M}} \left( u_{4}' \right)^{p-1} \geqslant \lambda \left( u_{4}' \right)^{p-1}; \\
u_{4}' \bigg |_{\partial M} & = (\tilde{\varphi} + C' ) \bigg |_{\partial M} = C' \geqslant \max_{\partial M} \phi \geqslant \phi.
\end{align*}
It follows that $ u_{4}' $ is a super-solution of (\ref{dirichlet:eqn1}) with the fixed $ \lambda > 0 $. With this $ \lambda $, we check from (\ref{dirichlet:eqn12}) that
\begin{align*}
-a\Delta_{g} u_{4} + S_{g} u_{4} & \leqslant -a\Delta_{g} u_{4} + \max_{\bar{M}} S_{g} u_{4} \leqslant -a\Delta_{g} u_{4} + \max_{\bar{M}} C u_{4} \\
& = 0 \leqslant \lambda u_{4}^{p-1}.
\end{align*}
Hence $ u_{4} $ is a sub-solution of (\ref{dirichlet:eqn1}) with the fixed $ \lambda > 0 $. Furthermore, (\ref{dirichlet:eqn14}) implies $ 0 < u_{4} \leqslant u_{4}' $. Since $ u_{4}, u_{4}' \in \calC^{\infty}(M) \cap \calC^{0}(\bar{M}) $, it follows from Theorem \ref{iteration:thm1} that there exists a real, positive solution $ u \in \calC^{\infty}(M) $ of (\ref{dirichlet:eqn1}) with the fixed $ \lambda > 0 $. In addition, $ 0 < u_{4} \leqslant u \leqslant u_{4}' $.
\end{proof}
\begin{remark}\label{dirichlet:re2}
When $ \tilde{\eta}_{1} = 0 $, the Dirichlet boundary value problem (\ref{dirichlet:eqn1}) is only solvable when $ \phi = 0 $ on $ \partial M $, due to Fredholm dichotomy. Instead, (\ref{dirichlet:eqn1}) admits a positive, smooth solution with some $ \lambda < 0 $ when $ \tilde{\eta}_{1} = 0 $. Since $ \tilde{\eta}_{1} = 0 $ implies $ \eta_{1} < 0 $, due to Proposition \ref{dirichlet:prop3}).
\end{remark}
\medskip

\section{Prescribed Scalar Curvature with Dirichlet Boundary Condition under Conformal Deformation}
Recall the Kazdan and Warner's ``Trichotomy Theorem" .
\begin{theorem}\label{scalar:thm0}\cite{KW3}, \cite{KW4}, \cite{KW}
Let $ M $ be a closed manifold, $ \dim M \geqslant 3 $.

(i) If $ M $ admits some Riemannian metric $ g $ whose scalar curvature function is nonnegative and not identically zero, then any function $ S \in \calC^{M} $ is realized as the scalar curvature function of some metric on $ M $; 

(ii) If $ M $ admits some Riemannian metric $ g $ with vanishing scalar curvature, and does not satisfy the condition in (i), then a function $ S \in \calC^{M} $ is a scalar curvature function of some metric if and only if either $ f < 0 $ at some point of $ M $, or $ f \equiv 0 $;

(iii) If $ M $ does not satisfy both conditions in (i) and (ii), then a function $ S \in \calC^{\infty}(M) $ is a scalar curvature function of some metric if and only if $ f < 0 $ at some point of $ M $.
\end{theorem}
In this section, we show results similar to the ``Trichotomy Theorem" on $ (\bar{M}, g) $ by restricting on the conformal class $ [g] $. Note that the results of Kazdan and Warner may not be obtained under pointwise conformal change. In particular, we consider the prescribed scalar curvature problem under conformal deformation of a Riemannian metric on $ \bar{M} $: given a smooth function $ S \in \calC^{\infty}(\bar{M}) $ and a positive function $ \phi \in \calC^{\infty}(\partial M) $, under which conditions should $ S $ satisfy so that there exists a metric $ \tilde{g} $ conformal to $ g $ with scalar curvature $ S $ on $ \bar{M} $, meanwhile the induced metric on $ \partial M $ is determined by $ \phi $. It is equivalent to the existence of a positive smooth function $ u $ satisfying
\begin{equation}\label{scalar:eqn1}
\begin{split}
-a\Delta_{g} u + S_{g} u & = Su^{p-1} \; {\rm in} \; M; \\
u = \phi \; {\rm on} \; \partial M.
\end{split}
\end{equation}
For the rest of this section, we discuss the solvability of (\ref{scalar:eqn1}) by imposing conditions on $ S $. In \cite{KW2}, results for closed manifolds and open manifolds within a conformal class are given. On closed manifolds $ (M, g) $, the restrictions for the prescribed scalar curvature is more restrictive. For example, they show that for $ S \in \calC^{\infty}(M) $ to be a scalar curvature with respect to some metric under conformal change, they must assume that $ S $ is negative everywhere when the first eigenvalue $ \lambda_{1} $ of $ \Box_{g} $ on closed manifold $ (M, g) $ is negative; $ S $ should be either identically zero, or change sign with $ \int_{M} S \dvol < 0 $ when $ \lambda_{1} = 0 $; a more delicate obstruction was given in \cite{KW2} when the manifold is the $ n $-sphere with standard metric. They have very few results for prescribed scalar curvature function on closed manifolds with positive first eigenvalue of conformal Laplacian within the conformal class $ [g] $. However, the results here on compact manifolds $ (\bar{M}, g) $ with smooth boundary require loose restrictions than in closed manifolds, since the conformal deformation of a metric is not only reflected and solely determined by the scalar curvature, but also reflected by the boundary behavior. This is also probably due to the fact that the constant functions are not eigenfunctions of Laplace-Beltrami operator. 
\medskip

We will show that when $ \eta_{1} > 0 $, our restriction on the prescribed scalar curvature in the interior $ M $ outperform the corresponding case in ``Trichotomy Theorem", even with the restriction that the metric should be within a conformal class; when $ \eta_{1} = 0 $, our restriction on the prescribed scalar curvature within a conformal class is slightly weaker than the ``Trichotomy Theorem", but better than the corresponding results on closed manifolds with pointwise conformal change; it is more restrictive when we discuss the case $ \eta_{1} < 0 $. 
\medskip

As in the previous section, our results are classified by the sign of $ \eta_{1} $, the first eigenvalue of conformal Laplacian with Robin boundary condition.

We consider the case $ \eta_{1} < 0 $ first, which is historically the easy case. Without loss of generality, we may assume that $ h_{g} > 0 $ everywhere on $ \partial M $ due to Theorem \ref{pre:thm3}. Note also that the sign of $ \eta_{1} $ is unchanged under conformal change by Proposition \ref{dirichlet:prop1}. Due to the characterization of $ \eta_{1} $,
\begin{equation*}
\eta_{1} = \inf_{u \in H^{1}(M, g) } \frac{\int_{M} a \lvert \nabla_{g} u \rvert^{2}\dvol + \int_{M} S_{g} u^{2} \dvol + \int_{\partial M} \frac{2}{p-2} h_{g} u^{2} dS}{\int_{M} u^{2} \dvol}
\end{equation*}
we conclude that $ S_{g} < 0 $ somewhere, when $ h_{g} > 0 $ everywhere on $ \partial M $. Therefore, it is natural to consider the candidate of prescribed scalar curvature $ S \in \calC^{\infty}(\bar{M}) $ to be negative somewhere, provided that $ \eta_{1} < 0 $. We start with a simple case.
\begin{theorem}\label{scalar:thm1}
Let $ (\bar{M}, g) $ be a compact Riemannian manifold with smooth boundary $ \partial M $. Let $ \phi > 0 $ be a smooth function on $ \partial M $. Given a function $ S \in \calC^{\infty}(\bar{M}) $ that is negative everywhere on $ \bar{M} $. If $ \eta_{1} < 0 $, then there exists a metric $ \tilde{g} = u^{p-2} g $ conformal to $ g $ such that the positive function $ u \in \calC^{\infty}(M) $ and $ u = \phi $ on $ \partial M $.
\end{theorem}
\begin{proof}
Since $ \eta_{1} < 0 $, the eigenvalue problem
\begin{equation}\label{scalar:eqn2}
-a\Delta_{g} \varphi + S_{g} \varphi = \eta_{1} \varphi \; {\rm in} \; M, \frac{\partial \varphi}{\partial \nu} + \frac{2}{p-2} h_{g} \varphi = 0 \; {\rm on} \; \partial M
\end{equation}
admits a positive, smooth solution $ \varphi $ on $ \bar{M} $. Any positive scaling $ \delta \varphi $ also solves (\ref{scalar:eqn1}). Choosing $ \delta > 0 $ small enough so that
\begin{equation*}
\eta_{1} \min_{\bar{M}} \varphi \leqslant \delta^{p-2} \min_{\bar{M}} \left( S \varphi^{p-1} \right), \delta \max_{\partial M} \varphi \leqslant \min_{\partial M} \phi.
\end{equation*}
Fix this $ \delta > 0 $ and denote $ u_{5} = \delta \varphi $. It follows from the choice of $ \delta $ that
\begin{align*}
-a\Delta_{g} u_{5} + S_{g} u_{5} & = \eta_{1} \delta \varphi \leqslant \eta_{1} \delta \min_{\bar{M}} \varphi \leqslant  \delta^{p-1} \min_{\bar{M}} \left( S \varphi^{p-1} \right) \leqslant S u_{5}^{p-1} \; {\rm in} \; M; \\
u_{5} & = \delta \varphi \leqslant \delta \max_{\partial M} \varphi \leqslant \min_{\partial M} \phi \leqslant \phi \; {\rm on} \; \partial M.
\end{align*}
It follows that the positive function $ u_{5} \in \calC^{\infty}(M) \cap \calC^{0}(\bar{M}) $ is a sub-solution of (\ref{scalar:eqn1}) with the choice of $ S $ in the statement. Choose a constant $ C > 0 $ large enough so that
\begin{equation}\label{scalar:eqn3}
C \geqslant \max_{\bar{M}} u_{5}, C \geqslant \max_{\partial M} \phi, \min_{\bar{M}} S_{g} \geqslant \max_{S} C^{p-2}.
\end{equation}
This can be done since $ S < 0 $ everywhere by assumption. Set $ u_{5}' = C $. By (\ref{scalar:eqn3}) we have
\begin{align*}
-a\Delta_{g} u_{5}' + S_{g} u_{5}' & = S_{g} C \geqslant \max_{S} C^{p-1} \geqslant S \left( u_{5}' \right)^{p-1} \; {\rm in} \; M; \\
u_{5}' & = C \geqslant \max_{\partial M} \phi \geqslant \phi \; {\rm on} \; \partial M.
\end{align*}
It follows that $ u_{5}' \geqslant u_{5} > 0 $ is a super-solution of (\ref{scalar:eqn1}) with this $ S \in \calC^{\infty}(\bar{M}) $. Clearly constant function is smooth. By Corollary \ref{iteration:cor1}, we conclude that there exist a positive function $ u \in \calC^{\infty}(M) $ solves (\ref{scalar:eqn1}). In addition, $ 0 < u_{5} \leqslant u \leqslant u_{5}' $.
\end{proof}
\begin{remark}
Unfortunately, we are not able to show that any function that is negative somewhere can be realized as a prescribed scalar curvature function with some metric after pointwise conformal change. But the statement holds if we replace the pointwise conformal change by conformally equivalent metrics. The proof is fairly easy and we omit the details. Similar arguments follow in the cases of $ \eta_{1} = 0 $ and $ \eta_{1} > 0 $.
\end{remark}
\medskip

We now move to the case $ \eta_{1} = 0 $. A result of Kazdan and Warner \cite{KW2} shows that when the first eigenvalue of $ \Box_{g} $ on closed manifold is zero, necessary condition of a given function $ S $ to be a prescribed scalar curvature of some metric under conformal change is either $ S \equiv 0 $ or $ S $ must change sign. On compact manifolds with boundary, it is reasonable to conjecture a loose restriction on $ S $: the necessary condition of the existence of the solution of (\ref{scalar:eqn1}) with some function $ S \in \calC^{\infty}(\bar{M}) $ is either (i) $ S \equiv 0 $; or (ii) $ S $ must be negative somewhere. In particular, we can show below that $ S < 0 $ everywhere on $ \bar{M} $ can be realized as  a prescribed scalar curvature function in this case. Heuristically, we can see these as follows:

If $ S \equiv 0 $, then (\ref{scalar:eqn1}) admits a positive, smooth solution due to Theorem \ref{dirichlet:thm2}. When $ S \not\equiv 0 $, consider the characterization of the first eigenvalue of $ \Box_{g} $ with respect to Robin condition:
\begin{equation*}
0 = \eta_{1} = \inf_{u \in H^{1}(M, g) } \frac{\int_{M} a \lvert \nabla_{g} u \rvert^{2}\dvol + \int_{M} S_{g} u^{2} \dvol + \int_{\partial M} \frac{2}{p-2} h_{g} u^{2} dS}{\int_{M} u^{2} \dvol}
\end{equation*}
We may assume $ h_{g} > 0 $ everywhere on $ \partial M $, due to Theorem \ref{pre:thm3}. Thus if $ S_{g} > 0 $ everywhere, then the eigenfunction $ \varphi $ with respect to $ \eta_{1} = 0 $ must satisfy
\begin{equation*}
0 < \int_{M} a \lvert \nabla_{g} \varphi \rvert^{2}\dvol + \int_{\partial M} \frac{2}{p-2} h_{g} \varphi^{2} dS = \int_{M} -S_{g} \varphi^{2} \dvol < 0.
\end{equation*}
Contradiction. Our first result below amounts to the given function $ S \in \calC^{\infty}(\bar{M}) $ which changes sign, with some extra restriction on $ S $. Given any constant $ \gamma > 0 $, denote
\begin{equation}\label{scalar:eqn4}
U_{\gamma}(\partial M) = \lbrace x \in \bar{M} : \text{dist}(x, \partial M) < \gamma \rbrace.
\end{equation}
Let $ \bar{U}_{\gamma}(\partial M) $ be the closure of the set in (\ref{scalar:eqn4}).
\begin{theorem}\label{scalar:thm2}
Let $ (\bar{M}, g) $ be a compact Riemannian manifold with smooth boundary $ \partial M $. Let $ \phi > 0 $ be a smooth function on $ \partial M $. Given a function $ S \in \calC^{\infty}(\bar{M}) $ which is negative on $ \bar{U}_{\gamma}(\partial M) $ for some $ \gamma > 0 $. If $ \eta_{1} = 0 $, then there exists a metric $ \tilde{g} = u^{p-2} g $ conformal to $ g $ and a positive constant $ c > 0 $ such that the positive function $ u \in \calC^{\infty}(M) $ and $ u = c\phi $ on $ \partial M $.
\end{theorem}
\begin{proof} Pick up a function $ \Phi' \equiv 1 $ on $ \bar{M} $. Denote $ \phi' = \Phi' \bigg |_{\partial M} = 1 $ on $ \partial M $. Since $ \eta_{1} = 0 $, we conclude by Theorem \ref{dirichlet:thm1} that there exist a metric $ g_{1} = v^{p-2} g $, where $ v $ is positive and smooth in $ M $, such that $ S_{g_{1}} = 0 $ with $ v = \phi' \equiv 1 $ on $ \partial M $. Clearly, the first eigenvalue of $ \Box_{g_{1}} $ is also zero, due to Proposition \ref{dirichlet:prop1}. Due to Proposition \ref{dirichlet:prop3}, the first eigenvalue of $ \Box_{g_{1}} $ with respect to Dirichlet condition, still denoted by $ \tilde{\eta}_{1} $, is positive. We have
\begin{equation}\label{scalar:eqn5}
-a\Delta_{g_{1}} \tilde{\varphi} = \tilde{\eta}_{1} \tilde{\varphi} \; {\rm in} \; M, \tilde{\varphi} = 0 \; {\rm on} \; \partial M.
\end{equation}
Fix a constant $ K $ which satisfies
\begin{equation*}
K \geqslant \max_{\partial M} \phi.
\end{equation*}
Note that the set $ U_{1}: = M \backslash U_{\gamma}(\partial M) $ is a closed set. Define
\begin{equation}\label{scalar:eqn6}
u_{6} : = \tilde{\varphi} + K, K_{1} = \min_{U_{1}} \tilde{\varphi}
\end{equation}
There exists a constant $ \beta > 0 $, small enough, such that
\begin{equation}\label{scalar:eqn7}
\tilde{\eta}_{1} K_{1} \geqslant \beta \cdot \max_{\bar{M}} S \max_{\bar{M}} u_{6}^{p-1}.
\end{equation}
It follows $ u_{6} $ satisfies
\begin{align*}
\tilde{\eta}_{1} \tilde{\varphi} & \geqslant \tilde{\eta}_{1} K_{1} \geqslant \beta \cdot \max_{\bar{M}} S \max_{\bar{M}} u_{6}^{p-1} \geqslant \beta S u_{6}^{p-1} \; {\rm in} \; U_{1}; \\
\tilde{\eta}_{1} \tilde{\varphi} & \geqslant \beta S u_{6}^{p-1} \; {\rm in} \bar{U}_{\gamma}(\partial M); \\
\Rightarrow \tilde{\eta}_{1} \tilde{\varphi} & \geqslant  \beta S u_{6}^{p-1} \; {\rm in} \; M.
\end{align*}
The first inequality above is due to (\ref{scalar:eqn7}); the second inequality above is due to the assumption that $ S < 0 $ on $ \bar{U}_{\gamma}(\partial M) $. Note that the inequality on last line above also holds for all smaller $ \beta > 0 $. Thus by (\ref{scalar:eqn5}), we have
\begin{align*}
-a\Delta_{g_{1}} u_{6} & = -a\Delta_{g_{1}} \tilde{\varphi} = \tilde{\eta}_{1} \tilde{\varphi} \geqslant  \beta S u_{6}^{p-1} \geqslant \beta' S u_{6}^{p-1} \; {\rm in} \; M; \\
u_{6} & = K \geqslant \max_{\partial M} \phi \geqslant \phi \geqslant \phi \; {\rm on} \; \partial M
\end{align*}
holds for all $ \beta' \leqslant \beta $. Clearly $ u_{6} > 0 $ on $ \bar{M} $ and $ u_{6} \in \calC^{\infty}(M) \cap \calC^{0}(\bar{M}) $. Thus $ u_{6} $ is a super-solution of (\ref{scalar:eqn1}) with $ \beta S $ on right side, with some scale $ \beta > 0 $. To construct the sub-solution, we choose any constant $ C > 0 $ and consider the PDE
\begin{equation}\label{scalar:eqn8}
-a\Delta_{g_{1}} u + Cu = 0 \; {\rm in} \; M, u = \phi \: {\rm on} \; \partial M.
\end{equation}
By the same argument in previous results, (\ref{scalar:eqn8}) admits a positive, smooth solution $ u $ on $ \bar{M} $. We choose a constant $ 0 <  \delta < 1 $ small enough so that
\begin{equation*}
\delta \max_{\bar{M}} u \leqslant K.
\end{equation*}
Denote $ u_{6}' : = \delta u $. It follows that when $ \beta > 0 $ is small enough, 
\begin{align*}
-a\Delta_{g_{1}} u_{6}' & = -Cu_{6}' \leqslant \beta \min_{\bar{M}} S \left(u_{6}'\right)^{p-1} \; {\rm in} \; M; \\
u_{6}' & = \delta \phi \leqslant \phi \; {\rm on} \; \partial M.
\end{align*}
Now fix $ \beta > 0 $ so that both inequalities for $ u_{6} $ and $ u_{6}' $ are satisfied, respectively. It follows that $ u_{6}' \in \calC^{\infty}(M) \cap \calC^{0}(\bar{M}) $ is a sub-solution of (\ref{scalar:eqn1}) with $ \beta S $. In addition, $ 0 < u_{6}' \leqslant u_{6} $. Due to Corollary \ref{iteration:cor1}, we conclude that there exists a positive, smooth function $ w $ satisfies
\begin{equation*}
-a\Delta_{g_{1}} w = \beta S w^{p-1} \; {\rm in} \; M, w = \phi \; {\rm on} \; \partial M, 0 < u_{6}' \leqslant w \leqslant u_{6}.
\end{equation*}
It is equivalent to say that there exists a metric $ g_{2} = w^{p-2} g_{1} $ with scalar curvature $ \beta S $ in $ M $, and $ w = \phi $ on $ \partial M $. Since $ g_{1} = v^{p-2} g $, we denote
\begin{equation*}
u = (wv)^{p-2} \; {\rm on} \; \bar{M}.
\end{equation*}
It follows that $ g_{2} = u^{p-2} g $ with scalar curvature $ \beta S $ and $ u = \phi \cdot 1 = \phi $ on $ \partial M $, i.e. $ u $ solves
\begin{equation*}
-a\Delta_{g} u + S_{g} u = \beta S u^{p-1} \; {\rm in} \; M, u = \phi \; {\rm on} \; \partial M.
\end{equation*}
Let $ c > 0 $ be a constant with
\begin{equation*}
c^{p-2} = \beta.
\end{equation*}
Define $ \tilde{u} = cu $, it follows by a direct computation that
\begin{equation*}
-a\Delta_{g} \tilde{u} + S_{g} \tilde{u} = S \tilde{u}^{p-1} \; {\rm in} \; M, \tilde{u} = c\phi \; {\rm on} \; \partial M.
\end{equation*}
Denote $ \tilde{g} = \tilde{u}^{p-2} g $, which is a scaling of $ g_{2} $, it follows that $ \tilde{g} $ has scalar curvature $ S $ in $ M $. The positive, smooth function $ \tilde{u} $ is the desired function.
\end{proof}
\medskip

Next result concerns the smooth function $ S \in \calC^{\infty}(\bar{M}) $ that is negative everywhere. This is easier than the case in Theorem \ref{scalar:thm2}.
\begin{corollary}\label{scalar:cor1}
Let $ (\bar{M}, g) $ be a compact Riemannian manifold with smooth boundary $ \partial M $. Let $ \phi > 0 $ be a smooth function on $ \partial M $. Given a function $ S \in \calC^{\infty}(\bar{M}) $ which is negative everywhere. If $ \eta_{1} = 0 $, then there exists a metric $ \tilde{g} = u^{p-2} g $ conformal to $ g $ such that the positive function $ u \in \calC^{\infty}(M) $ and $ u = \phi $ on $ \partial M $.
\end{corollary}
\begin{proof} Since $ \eta_{1} $, the eigenfunction $ \varphi > 0 $ on $ \bar{M} $ satisfies
\begin{equation}\label{scalar:eqn9}
-a\Delta_{g} \varphi + S_{g} \varphi = 0 \; {\rm in} \; M, \frac{\partial \varphi}{\partial \nu} + \frac{2}{p-2} h_{g} \varphi = 0 \; {\rm on} \; \partial M.
\end{equation}
Any scale $ \delta \varphi $ also solves (\ref{scalar:eqn9}), thus we choose $ \delta > 0 $ so that
\begin{equation*}
\delta \min_{\partial M} \varphi \geqslant \max_{\partial M} \phi.
\end{equation*}
Fix this $ \delta > 0 $ and denote $ u_{7}' = \delta \varphi $, we can see that
\begin{equation*}
-a\Delta_{g} u_{7}' + S_{g} u_{7}' \geqslant S \left( u_{7}' \right)^{p-1} \; {\rm in} \; M, u_{7}' \geqslant \phi \; {\rm on} \; M.
\end{equation*}
Hence the positive function $ u_{7}' \in \calC^{\infty}(M) \cap \calC^{0}(\bar{M}) $ is a super-solution of (\ref{scalar:eqn1}) with the given negative function $ S $. The choice of sub-solution and the rest of the argument are exactly the same as in Theorem \ref{scalar:thm2}.
\end{proof}
\begin{remark}\label{scalar:re1}
Combining results of Theorem \ref{dirichlet:thm1}, Theorem \ref{scalar:thm2} and Corollary \ref{scalar:cor1}, we conclude that either $ S \equiv 0 $, $ S $ negative on $ \bar{M} $ or $ S $ negative near the boundary of $ \bar{M} $ can be a prescribed scalar curvature of some metric under conformal change.
\end{remark}
\medskip

We now discuss the case $ \eta_{1} > 0 $. Analogously, when $ (M, g) $ is a closed manifold with $ \lambda_{1} > 0 $, there is a topological obstruction on closed spin manifolds with even dimensions, which says that the manifolds that admit positive scalar curvatures must have zero $ \hat{A} $-genus, see e.g. \cite{Gromov}. In \cite{KW2}, the obstruction on $ n $-sphere with standard metric was given, and Kazdan and Warner stated that this should be the only obstruction on manifolds with positive scalar curvature. Some recent argument with respect to prescribing Morse scalar curvature was given, see \cite{MaMa}. Other than these results, very little is known. This is probably because of the fact that the evolution of conformal deformation is solely reflected and determined by the scalar curvature. On compact manifolds $ (\bar{M}, g) $ with smooth boundary, we show below that any function $ S \in \calC^{\infty}(\bar{M}) $ can be a scalar curvature of some metric under conformal change when $ \eta_{1} > 0 $.

The first result concerns the case $ S > 0 $ everywhere on $ \bar{M} $ provided that $ \eta_{1} > 0 $.
\begin{theorem}\label{scalar:thm3}
Let $ (\bar{M}, g) $ be a compact Riemannian manifold with smooth boundary $ \partial M $. Let $ \phi > 0 $ be a smooth function on $ \partial M $. Given a function $ S \in \calC^{\infty}(\bar{M}) $ which is positive everywhere. If $ \eta_{1} > 0 $, then there exists a metric $ \tilde{g} = u^{p-2} g $ conformal to $ g $ and a positive constant $ c $ such that the positive function $ u \in \calC^{\infty}(M) $ and $ u = c\phi $ on $ \partial M $.
\end{theorem}
\begin{proof} Since $ \eta_{1} >0 $, the corresponding eigenfunction $ \varphi > 0 $ solves
\begin{equation}\label{scalar:eqn10}
-a\Delta_{g} \varphi + S_{g}\varphi = \eta_{1}\varphi \; {\rm in} \; M, \frac{\partial \varphi}{\partial \nu} + \frac{2}{p-2} h_{g} \varphi = 0 \; {\rm on} \; \partial M.
\end{equation}
Scaling $ \varphi \mapsto \delta_{1} \varphi $ with an appropriate choice of $ \delta_{1} > 0 $ small enough such that
\begin{equation*}
\eta_{1} \min_{\bar{M}} \varphi \geqslant \delta_{1}^{p-2} \max_{\bar{M}} S \max_{\bar{M}} \varphi^{p-1}.
\end{equation*}
Fix this $ \delta_{1} > 0 $. Denote $ u_{8}' : = \delta_{1} \varphi $. It follows that there exists some small enough constant $ c > 0 $ such that
\begin{align*}
-a\Delta_{g} u_{8}' + S_{g} u_{8}' & = \delta_{1} \eta_{1} \varphi \geqslant \delta_{1} \eta_{1} \min_{\bar{M}} \varphi \geqslant \delta_{1}^{p-1} \max_{\bar{M}} S \max_{\bar{M}} \varphi^{p-1} \geqslant S \left( u_{8}' \right)^{p-1} \; {\rm in} \; M; \\
u_{8}' & \geqslant \delta_{1} \min_{\partial M} \varphi \geqslant c \max_{\partial M} \phi \geqslant c \phi \; {\rm on} \; \partial M.
\end{align*}
Fix this constant $ c $. The positive function $ u_{8}' $ is a super-solution of (\ref{scalar:eqn1}) with boundary condition $ c\phi $ on $ \partial M $. For sub-solution, we choose $ C > 0 $ such that
\begin{equation*}
C \geqslant \max_{\bar{M}} S_{g}.
\end{equation*}
We see from above that there exists a positive function $ u \in \calC^{\infty}(M) \cap \calC^{0}(\bar{M}) $ solves
\begin{equation}\label{scalar:eqn11}
-a\Delta_{g} u + Cu = 0 \; {\rm in} \; M, u = c\phi \; {\rm on} \; \partial M.
\end{equation}
Denote $ u_{8} = \delta_{2} u $ with $ 0 < \delta_{2} \ll 1 $ small enough such that
\begin{equation*}
\max_{\bar{M}} u_{8} \leqslant \min_{\bar{M}} u_{8}'.
\end{equation*}
We have
\begin{equation*}
-a\Delta_{g} u_{8} + S_{g} u_{8} \leqslant -a\Delta_{g} u_{8} + \max_{\bar{M}} S_{g} u_{8} \leqslant -a\Delta_{g} u_{8} + Cu_{8} = 0 \leqslant S u_{8}^{p-1} \; {\rm in} \; M, u_{8} \leqslant c\phi \; {\rm on} \; \partial M.
\end{equation*}
Thus $ u_{8} $ is a sub-solution of (\ref{scalar:eqn1}) with boundary condition $ c \phi $. Furthermore, $ 0 < u_{8} \leqslant u_{8}' $. We conclude by Corollary \ref{iteration:cor1} that there exists $ u \in \calC^{\infty}(M) $, $ 0 < u_{8} \leqslant u \leqslant u_{8}' $, that solves (\ref{scalar:eqn1}) with boundary condition $ c\phi $.
\end{proof}
\medskip

The next two results cover the cases $ S > 0 $ somewhere and $ S \leqslant 0 $ everywhere, provided that $ \eta_{1} > 0 $.
\begin{corollary}\label{scalar:cor2}
Let $ (\bar{M}, g) $ be a compact Riemannian manifold with smooth boundary $ \partial M $. Let $ \phi > 0 $ be a smooth function on $ \partial M $. Given a function $ S \in \calC^{\infty}(\bar{M}) $ which is positive somewhere. If $ \eta_{1} > 0 $, then there exists a metric $ \tilde{g} = u^{p-2} g $ conformal to $ g $ and a positive constant $ c $ such that the positive function $ u \in \calC^{\infty}(M) $ and $ u = c\phi $ on $ \partial M $.
\end{corollary}
\begin{proof}
The choice of super-solution $ u_{9}' $ of (\ref{scalar:eqn1}) with boundary condition $ c \phi $ for some $ c > 0 $ is the same as in Theorem \ref{scalar:thm3}. For sub-solution, we consider the positive solution $ u $ of 
\begin{equation}\label{scalar:eqn11a}
-a\Delta_{g} u + Cu = 0 \; {\rm in} \; M, u = c\phi \; {\rm on} \; \partial M.
\end{equation}
with $ C \geqslant \max_{\bar{M}} S_{g} $. Scaling $ u \mapsto \delta u $ for $ 0 < \delta \ll 1 $ small enough so that
\begin{equation*}
-C u \leqslant \delta^{p-2} \min_{\bar{M}} \left( S u^{p-1} \right).
\end{equation*}
This can be done when $ \min_{\bar{M}} S < 0 $. It is trivial if $ S \geqslant 0 $ everywhere. We choose $ \delta $ even smaller so that $ u_{9} : = \delta u \leqslant u_{9}' $ everywhere on $ \bar{M} $, as above. The rest arguments are the same.
\end{proof}
\begin{corollary}\label{scalar:cor3}
Let $ (\bar{M}, g) $ be a compact Riemannian manifold with smooth boundary $ \partial M $. Let $ \phi > 0 $ be a smooth function on $ \partial M $. Given a function $ S \in \calC^{\infty}(\bar{M}) $ which is nonpositive everywhere. If $ \eta_{1} > 0 $, then there exists a metric $ \tilde{g} = u^{p-2} g $ conformal to $ g $ and a positive constant $ c $ such that the positive function $ u \in \calC^{\infty}(M) $ and $ u = \phi $ on $ \partial M $.
\end{corollary}
\begin{proof} Due to the eigenvalue problem (\ref{scalar:eqn10}), the choice super-solution is trivial with $ c = 1 $ in previous two results. The choice of sub-solution is the same as in Corollary \ref{scalar:cor2}.
\end{proof}
\begin{remark}\label{scalar:re2}
Due to Theorem \ref{scalar:thm3}, Corollary \ref{scalar:cor2} and \ref{scalar:cor3}, we conclude that any smooth function can be a scalar curvature of some metric under conformal change. This is different from the closed manifold case. On closed manifolds $ (M, g) $, $ S \in \calC^{\infty}(M) $ is a scalar curvature under conformal change if some positive function $ u \in \calC^{\infty}(M) $ solves
\begin{equation*}
-a\Delta_{g} u + S_{g} u = Su^{p-1} \; {\rm in} \; M.
\end{equation*}
We can assume $ S_{g} \equiv 1 $, due to the results of the Yamabe problem. Pairing both sides with $ u $, we have
\begin{equation*}
a \lVert \nabla_{g} u \rVert_{\calL^{2}(M, g)}^{2} + \lVert u \rVert_{\calL^{2}(M, g)}^{2} = \int_{M} S u^{p} \dvol.
\end{equation*}
It follows that $ S $ must be positive somewhere on closed manifolds. There are other obstructions on $ n $-sphere, see \cite{KW2}. The difference between closed manifolds and compact manifolds with boundary is due to the fact that the boundary behavior will compensate the deformation in the interior of $ \bar{M} $.
\end{remark}
\medskip

Similar to the constant curvature case in Theorem \ref{dirichlet:thethm}, we have the following result with respect to prescribed scalar curvature.
\begin{theorem}\label{scalar:thethm} Let $ (\bar{M}, g) $ be a compact manifold with smooth boundary $ \partial M $.

(i) If $ \eta_{1} < 0 $, any negative function $ S \in \calC^{\infty}(\bar{M}) $ is the prescribed curvature of some metric under conformal change in $ M $; meanwhile the $ ( n - 1) $-dimensional closed manifold $ \partial M $ admits a constant scalar curvature with the same conformal change;

(ii) If $ \eta_{1} = 0 $, any function $ S \in \calC^{\infty}(\bar{M}) $ that is negative in $ U_{\gamma}(\partial M) $ for any $ \gamma > 0 $ or $ S \equiv 0 $ is the prescribed curvature of some metric under conformal change in $ M $; meanwhile the $ ( n - 1) $-dimensional closed manifold $ \partial M $ admits a constant scalar curvature with the same conformal change;

(iii) If $ \eta_{1} > 0 $, any function $ S \in \calC^{\infty}(\bar{M}) $ is the prescribed curvature of some metric under conformal change in $ M $; meanwhile the $ ( n - 1) $-dimensional closed manifold $ \partial M $ admits a constant scalar curvature with the same conformal change;
\end{theorem}
\begin{proof} We prove (iii) above, the rest are almost the same. Let
\begin{equation*}
\imath : \partial M \hookrightarrow \bar{M}
\end{equation*}
be the inclusion map and thus $ \imath^{*} g $ is the induced metric on $ \partial M $. When $ n \geqslant 4 $, there exists a smooth function $ \tilde{\phi} > 0 $ on the $ (n - 1) $-dimensional manifold $ \partial M $ such that
\begin{equation*}
g_{0} = \tilde{\phi}^{\frac{n - 1 + 2}{n - 1 - 2}} \left(\imath^{*} g \right) = \tilde{\phi}^{\frac{n + 1}{n - 3}} \left(\imath^{*} g \right) : = \tilde{\phi}^{p'} \left(\imath^{*} g \right), n \geqslant 4
\end{equation*}
admits a constant scalar curvature on $ \partial M $, due to the result of the Yamabe problem. When $ n = 3 $, there exists a smooth function $ f $ on $ \partial M $ such that
\begin{equation*}
g_{0}' = e^{2f} \left(\imath^{*} g \right), n = 3
\end{equation*}
admits a constant Gaussian curvature, due to the uniformization theorem. Denote
\begin{equation}\label{scalar:eqn12}
\phi = \begin{cases} \tilde{\phi}^{\frac{p'}{p - 2}}, & n \geqslant 4 \\ \left(e^{2f} \right)^{\frac{1}{p - 2}}, & n = 3 \end{cases}.
\end{equation}
Fix this $ \phi \in \calC^{\infty}(\partial M ) $ when $ n \geqslant 3 $. Given any function $ S \in \calC^{\infty}(\bar{M}) $, we apply either Theorem \ref{scalar:thm3}, Corollary \ref{scalar:cor2} or \ref{scalar:cor3} to conclude that (\ref{scalar:eqn1}) admits a positive solution $ u \in \calC^{\infty}(M) $ with boundary condition $ c\phi $ on $ \partial M $, for some constant $ c > 0 $. It follows that the metric
\begin{equation*}
\tilde{g} = u^{p-2} g
\end{equation*}
admits the scalar curvature $ S_{\tilde{g}} = S $ in the interior $ M $. Furthermore,
\begin{equation*}
\imath^{*} \tilde{g} = \imath^{*} \left (u^{p-2} g \right) = c^{p-2}\phi^{p-2} \imath^{*} g = \begin{cases} c_{1} \tilde{\phi}^{p'} \imath^{*} g = g_{0}, & n \geqslant 4 \\ c_{2}e^{2f} \imath^{*} g = g_{0}', & n = 3 \end{cases}
\end{equation*}
for some positive constants $ c_{1} $ and $ c_{2} $. Hence the induced metric $ \imath^{*} \tilde{g} $ admits a constant scalar curvature on $ (n - 1) $-dimensional closed manifold $ \partial M $ simultaneously.
\end{proof}

\bibliographystyle{plain}
\bibliography{Yamabessd}

\end{document}